\documentclass[preprint, 5p]{elsarticle}

\usepackage{lineno,hyperref}
\modulolinenumbers[5]

\journal{System and Control Letters}

\usepackage[T1]{fontenc}
\usepackage{amsmath}
\usepackage{amsfonts}
\usepackage{amssymb}
\usepackage{xcolor}
\usepackage{amsthm}

\usepackage{verbatim}
\usepackage{float}

\usepackage{enumitem}

\usepackage{graphicx}
\usepackage{tikz}
\usetikzlibrary{arrows}
\tikzstyle{int}=[draw, fill=blue!20, minimum size=2em]
\tikzstyle{init} = [pin edge={to-,thin,black}]
\usepackage{pgfplots}

\newtheorem{Thm}{Theorem}[section]
\newtheorem{Lem}[Thm]{Lemma}

\newtheorem*{Thm*}{Theorem}
\newtheorem{Cor}[Thm]{Corollary}

\theoremstyle{definition}
\newtheorem{Def}[Thm]{Definition}
\newtheorem{Rem}[Thm]{Remark}

\newcommand{\R}{\mathbb{R}}
\newcommand{\N}{\mathbb{N}}
\newcommand{\cB}{\mathcal{B}}
\newcommand{\cC}{\mathcal{C}}

\newcommand{\cF}{\mathcal{F}}
\newcommand{\cL}{\mathcal{L}}
\newcommand{\cM}{\mathcal{M}}
\newcommand{\cV}{\mathcal{V}}

\newcommand{\vp}{\varphi}
\newcommand{\ve}{\varepsilon}
\newcommand{\rp}{\mathbb{R}_{\geq 0}}

\newcommand{\setdef}[2]{\left\{\, #1\, \left|\, \vphantom{#1} #2\, \right.\right\}}

\newenvironment{smallpmatrix}
{\left(\begin{smallmatrix}}
{\end{smallmatrix}\right)}

\newenvironment{smallbmatrix}
{\left[\begin{smallmatrix}}
{\end{smallmatrix}\right]}

\DeclareMathOperator{\im}{im}

\DeclareMathOperator{\diag}{diag}
\DeclareMathOperator{\Gl}{\mathbf{Gl}}
\DeclareMathOperator{\rk}{\rm rk}

\begin{document}

\begin{frontmatter}

\title{Internal dynamics of multibody systems\tnoteref{Support}}
\tnotetext[Support]{This work was supported by the German Research Foundation (Deutsche Forschungsgemeinschaft) via the grant BE 6263/1-1. \\
\textit{Email adress:} \textbf{lanza@math.upb.de} (Lukas Lanza)}


\author[mymainaddress]{Lukas Lanza}


\address[mymainaddress]{Institut f\"ur Mathematik, Universit\"at Paderborn, Warburger Str. 100, 33098 Paderborn, Germany}

\begin{abstract}
We consider nonlinear multibody systems and present a suitable set of coordinates for the internal dynamics which allow to decouple the internal dynamics without the need to compute the Byrnes-Isidori form. 
Furthermore, we derive sufficient conditions for the system parameters such that the internal dynamics of a class of systems with constant mass matrix are bounded-input, bounded-output stable.
\end{abstract}


\end{frontmatter}


\section{Introduction}
In the last two decades it turned out that for the purpose of high-gain based output tracking the so called funnel controller developed in~\cite{IlchRyan02b} and generalized to nonlinear systems with arbitrary relative degree in~\cite{BergLe18a} is a powerful tool. 
However, both necessitate stability of the internal dynamics in a certain sense.
Although there is progress in tracking control of systems with unstable internal dynamics, see e.g. \cite{Berg20a, BergLanz20, Seif12b, Seif14}, most controllers require the internal dynamics to be bounded-input, bounded-output stable (minimum phase property), c.f. 
\cite{ByrnWill84, ByrnIsid84, KhalSabe87, MillDavi91, IlchRyan02b, SeifBlaj13, BergLe18a, BergReis18a, BergOtto19} 
and see also the survey~\cite{IlchRyan08}, and for a detailed introduction to the concept of minimum phase~\cite{IlchWirt13}, and the references therein, resp. Therefore, in order to verify applicability of a certain controller it is necessary to decouple the internal dynamics, e.g. via the Byrnes-Isidori form as in~\cite{Isid95}, and investigate its stability. 
However, the computation of the Byrnes-Isidori form of a nonlinear multibody system is often a challenging task. In the present work we combine the idea of the Byrnes-Isidori form with a novel approach to decouple the internal dynamics without the need to compute the Byrnes-Isidori form explicitly. This results in a representation of the internal dynamics in terms of the internal variables and the system's output. \\
Moreover, we present sufficient conditions on the system parameters such that the internal dynamics are bounded-input, bounded-output stable. These conditions are independent of the representation of the internal dynamics and can be verified without their explicit decoupling. Hence, applicability of a controller e.g as in~\cite{IlchRyan07, BergReis18a, BergLe18a} can be determined without decomposing the system but by investigation of the system's equations only. 
\ \\
This paper is organized as follows. We briefly recall the concepts of Lie derivatives, relative degree and the representation of a dynamical system in Byrnes-Isidori form in Section~\ref{Sec:System-class}. 
In Section~\ref{Sec:Representation-ID} we derive the representation of the internal dynamics in terms of the internal variables and the system's output. Furthermore, we give lemmata of existence and uniqueness, resp., concerning the novel structural ansatz for the internal dynamics. 
In Section~\ref{Sec:Stability-of-ID} we give an abstract stability result exploiting LaSalle's invariance principle presented in~\cite{LaSa60}. Furthermore, we consider nonlinear multibody systems without kinematic loop and provide sufficient conditions on the system parameters such that the internal dynamics are bounded-input, bounded-output stable.
We finish this paper with two illustrative examples in Section~\ref{Sec:Example}. 

\subsection{Nomenclature}
Throughout the present paper we will use the following notation:
$\rp  :=[0,\infty)$;
$\|x\|:= \sqrt{x^\top x}$ Euclidean norm of $x \in \R^n$;
$\| A \|:= \max_{\|x\| = 1} \| A x\|$ spectral norm of~$A \in \R^{m \times n}$;
$\Gl_n(\R)$  the group of invertible matrices in $\R^{n \times n}$;
$B_r(x):= \{ z \in \R^n \ \vline \ \|z-x\|<r \}$ open ball of radius $r > 0$, centred at $x \in \R^n$;
$\cC^k(V \to W)$  the set of $k$~times continuously differentiable functions~$f: V \to W$, $k \in \N$ and $V \subseteq \R^n$, $W \subseteq \R^m$;
$\cL^\infty(I \to  \R^n)$  the set of essentially bounded functions $f: I \to \R^n$,  $I \subseteq \R$ an interval;
$\| f \|_\infty := \sup_{t \in I}\|f(t)\|$ the supremum norm of $f \in \cL^\infty(I\to \R^n)$.

\section{System class, relative degree and Byrnes-Isidori form} \label{Sec:System-class}
In this section we briefly recall some basic concepts such as relative degree and the representation of a system in Byrnes-Isidori form.
We consider nonlinear multibody systems without kinematic loops which are modeled using generalized coordinates and are of the form
\begin{align} \label{eq:System}
\dot q(t) &= v(t), && q(0) = q^0 \in \R^n, \nonumber \\
M(q(t)) \dot v(t) & = f(q(t), v(t)) + B(q(t)) u(t), && v(0) = v^0 \in \R^n, \nonumber\\
y(t) &= h(q(t)),
\end{align}
where $M \in \cC( \R^n \to \Gl_n(\R))$ is the generalized mass matrix, $f \in \cC( \R^n \times \R^n \to \R^n )$ are the generalized forces, $B \in \cC( \R^n \to \R^{n \times m})$ is the distribution of the input and $h \in \cC^1( \R^n \to \R^m)$ is the measurement. 
The functions $u: \R_{\ge 0} \to \R^m$ are the inputs that exert an influence to system~\eqref{eq:System}, and $y: \R_{\ge 0} \to \R^m$ are the outputs that typically represent physically meaningful measurements of system~\eqref{eq:System}. Note, that the dimensions of the input and output coincide but we do not assume collocation, i.e., we do not assume~$h'(q) = B(q)^\top$.
For system~\eqref{eq:System} we introduce the state variables~$x_1 := q$, $x_2 := v$, setting~$x := (x_1^\top, x_2^\top)^\top$, and  
transform~\eqref{eq:System} into the system of first order ordinary differential equations
\begin{align} \label{eq:ODE}
\dot x(t)& = 
\underbrace{
\begin{pmatrix}
x_2(t) \\
 M(x_1(t))^{-1}f(x_1(t), x_2(t))
\end{pmatrix}}_{= : F(x_1(t), x_2(t))} \nonumber \\
&+ \underbrace{
\begin{bmatrix}
0 \\
M(x_1(t))^{-1}B(x_1(t))
\end{bmatrix}}_{=: G(x_1(t),x_2(t))} u(t),  && x(0) = x^0 \in \R^{2n} \nonumber \\
y(t) &= \tilde h(x_1(t), x_2(t)),
\end{align}
where~$\tilde h : \R^{2n} \to \R^m$ with~$\tilde h(x_1,x_2) = h(x_1)$.
In order to decouple the internal dynamics we invoke the Byrnes-Isidori form for~\eqref{eq:ODE}. To this end, recall the definition of the \textit{Lie derivative} of a function~$h$ along a vector field~$F$ at a point~$z \in U \subseteq \R^{2n}$, $U$ open
\begin{equation*}
(L_F \tilde h)(z) := \tilde h'(z) F(z),
\end{equation*}
where~$\tilde h'$ is the Jacobian of~$\tilde h$. We may successively define $L_F^k \tilde h = L_F(L_F^{k-1} \tilde h)$ with $L_F^0 \tilde h =\tilde h$. We denote with~$g_i(z)$ the columns of~$G(z)$ for $i = 1,...,m$ and define
\begin{equation*}
(L_G \tilde h)(z) := [(L_{g_1} \tilde h)(z),...,(L_{g_m} \tilde h)(z) ].
\end{equation*}
In accordance with~\cite[Sec. 5.1]{Isid95} we recall the concept of (strict) \textit{relative degree} for multi-input, multi-output systems.
\begin{Def}
System~\eqref{eq:ODE} has relative degree~$r \in \N$ on~$U \subseteq \R^{2n}$ open, if for all~$z \in U$ we have
\begin{equation} \label{rel-deg}
\begin{aligned}
 \forall\, k \in \{0,...,r-2 \}: \ (L_G^{} L_F^k \tilde h)(z) &= 0_{m \times m} \\
\text{\rm and} \ \tilde \Gamma(z) := (L_G^{} L_F^{r-1} \tilde h)(z) &\in \Gl_m(\R),
\end{aligned}
\end{equation}
where~$\tilde \Gamma: U \to \Gl_m(\R)$ is the high-gain matrix. 
\end{Def}
Now, if system~\eqref{eq:ODE} has relative degree~$r$ on an open set~$U \subseteq \R^{2n}$, then there exists a (local) diffeomorphism 
$\Phi: U \to W \subseteq \R^{2n}$, $W$ open, such that 
\begin{equation*}
\begin{pmatrix} \xi(t) \\ \eta(t) \end{pmatrix} = \Phi (x(t)),
\end{equation*}
with~$ \xi(t) \in \R^{rm}$, $\eta(t) \in \R^{2n-rm}$ transforms system~\eqref{eq:ODE} nonlinearly into \textit{Byrnes-Isidori form}
\begin{align} \label{eq:abstract-BIF}
y(t) &= \xi_1(t), \nonumber\\
\dot \xi_1(t) &= \xi_2(t), \nonumber \\
& \ \ \vdots \\
\dot \xi_{r-1}(t) &= \xi_r(t), \nonumber\\
\dot \xi_r(t) &= (L_F^r \tilde h)\big( \Phi^{-1}(\xi(t),\eta(t))\big) + \Gamma \big( \Phi^{-1}(\xi(t),\eta(t)) \big) u(t), \nonumber\\
\dot \eta(t) &= q(\xi(t),\eta(t)) + p(\xi(t),\eta(t)) u(t). \nonumber
\end{align}
The last equation in~\eqref{eq:abstract-BIF} represents the internal dynamics of system~\eqref{eq:ODE}. Note, that for~$r\cdot m < 2n$ system~\eqref{eq:ODE} has nontrivial interal dynamics.
\\

With the aid of Lie derivatives the diffeomorphism~$\Phi$ can be represented as
\begin{equation}\label{fct:abstract-BIF-representation}
\Phi(x) = \begin{pmatrix} \tilde h(x) \\ (L_F^{} \tilde h)(x) \\ \vdots \\ (L_F^{r-1} \tilde h)(x) \\ \tilde \phi_{1}(x) \\ \vdots \\ \tilde \phi_{2n-rm}(x) \end{pmatrix}, 
\ x \in U \subseteq \R^{2n}
\end{equation}
where $\tilde \phi_i: U \to \R$, $i=1,...,2n-rm$,
are such that~$\Phi'(x)$ is invertible for all~$ x \in U$. We recall $x = (x_1^\top, x_2^\top)^\top \in \R^{2n}$ and make the following assumption.
\begin{enumerate}[label=\textbf{(A\arabic*)},ref=(A\arabic*)]
\item \label{ass:hMB-regular} Given some open set $U_1 \subseteq \R^n$ and $H(x_1) := h'(x_1)$,
 we have $ \Gamma(x_1) := H(x_1) M(x_1)^{-1} B(x_1) \in \Gl_m(\R) $  for all $x_1 \in U_1$.
\end{enumerate}
\begin{Lem} \label{Lem:r-2} 
Consider system~\eqref{eq:ODE} and assume~\ref{ass:hMB-regular}. Then system~\eqref{eq:ODE} has relative degree~$r=2$ on $U := U_1 \times \R^n$.
\end{Lem}
\begin{proof}
Let~$x = (x_1^\top, x_2^\top)^\top \in U$.
Set~$H(x_1):= h'(x_1)$ and $\tilde H(x) := \tilde h'(x)$, then $\tilde H(x) = \begin{bmatrix} H(x_1) & 0 \end{bmatrix}$. Now, compute the Lie derivatives of~\eqref{eq:ODE} for~$x \in U$
\begin{align*}
(L_G \tilde h)(x) &= \begin{bmatrix} H(x_1) & 0 \end{bmatrix}  \begin{bmatrix} 0 \\ M(x_1)^{-1} B(x_1) \end{bmatrix} \\
&= 0_{m \times m}, \\
(L_F \tilde h)(x) &= \begin{bmatrix} H(x_1) & 0 \end{bmatrix}\begin{pmatrix} x_2 \\ M(x_1)^{-1} f(x_1,x_2) \end{pmatrix}  \\
&= H(x_1) x_2, \\
(L_G L_F \tilde h)(x) &= \begin{bmatrix} \frac{\partial}{\partial x_1} H(x_1) x_2 & H(x_1) \end{bmatrix} \begin{bmatrix} 0 \\ M(x_1)^{-1} B(x_1) \end{bmatrix} \\
&=: \Gamma(x_1) 
\end{align*}
where~$\Gamma(x_1) = H(x_1) M(x_1)^{-1} B(x_1)$ is invertible by assumption~\ref{ass:hMB-regular}. Therefore, according to~\eqref{rel-deg} system~\eqref{eq:ODE} has relative degree~$r=2$ on~$U$.
\end{proof}

\section{Representation of the internal dynamics} \label{Sec:Representation-ID}
In this section we introduce a novel structural ansatz for the internal dynamics and present a set of feasible coordinates to represent the internal dynamics of~\eqref{eq:ODE} without the need to compute the Byrnes-Isidori form explicitly.
We show the existence and feasibility of functions forming the novel representation of the internal dynamics in Lemma~\ref{Lem:existence-phi1/2}, and provide a particular choice and show its uniqueness in Lemma~\ref{Lem:phi2}. At the end of this section we present a representation of the internal dynamics which is completely determined by the system parameters.

\ \\
Via~\eqref{fct:abstract-BIF-representation} with~$r=2$ we obtain the following representation of the diffeomorphism~$\Phi$
\begin{equation*}
\begin{pmatrix} \xi \\ \eta \end{pmatrix} = \Phi(x) = \begin{pmatrix} h(x_1) \\ H(x_1) x_2 \\ \tilde \phi_1(x) \\ \vdots \\ \tilde \phi_{2n-2m}(x) \end{pmatrix}, \quad x = \begin{pmatrix} x_1 \\ x_2 \end{pmatrix} \in U,
\end{equation*}
where $\tilde \phi_i: U \to \R$, $i=1,...,2n-2m$ and, as in Lemma~\ref{Lem:r-2}, $U_1 \subseteq \R^n$ is an open set and~$U = U_1 \times \R^n$. \\
We make the following structural ansatz for the internal state $\eta = (\tilde \phi_1(x),...,\tilde \phi_{2n-2m}(x))^\top$
\begin{equation} \label{eq:eta-structure}
\eta = \begin{pmatrix} \eta_1 \\ \eta_2 \end{pmatrix} = \begin{pmatrix} \phi_1(x_1) \\ \phi_2(x_1) x_2 \end{pmatrix},
\end{equation}
where~$\phi_1 \in \cC^1( U_1 \to \R^{n-m})$ and~$\phi_2 \in \cC (U_1 \to \R^{(n-m) \times n})$. We highlight, that
\begin{equation}\label{eq:xi-structure}
\xi = \begin{pmatrix} \xi_1 \\ \xi_2 \end{pmatrix} = \begin{pmatrix} h(x_1) \\ H(x_1) x_2 \end{pmatrix}
\end{equation}
and~\eqref{eq:eta-structure} have similar structure, i.e., the internal dynamics are in the form of a mechanical system as well. Now, since $\Phi$ is a diffeomorphism we require its Jacobian to be invertible on~$U$
\begin{equation}
\label{eq:Phi-Jacobian}
\begin{aligned}
\forall\, x \in U: \
\Phi'(x) = \begin{bmatrix}
H(x_1) & 0 \\
* & H(x_1) \\
\phi'_1(x_1) & 0 \\
* & \phi_2(x_1)
\end{bmatrix} \in \Gl_{2n}(\R)
\iff \\
\forall \, x_1 \in U_1: \
\begin{bmatrix} H(x_1) \\ \phi'_1(x_1) \end{bmatrix}  \in \Gl_n(\R) \ \wedge \ 
\begin{bmatrix} H(x_1) \\ \phi_2(x_1) \end{bmatrix} \in \Gl_n(\R), 
\end{aligned}
\end{equation}
where $*$ is of the form $\tfrac{\partial}{\partial x_1} [\zeta(x_1) \cdot x_2]$, $\zeta: U_1 \to \R^{q \times n}$ with $q \in \N$ appropriate, resp. 
\ \\
We aim to investigate the internal dynamics of~\eqref{eq:ODE} without explicit appearance of the input~$u$. To this end, we seek for functions~$\tilde \phi_1(x),..., \tilde \phi_{2n-2m}(x)$ such that~$p(\cdot) = 0$ in equation~\eqref{eq:abstract-BIF}, i.e., $[L_G \tilde \phi_i](x) = 0$ for all $x \in U$, $i=1,...,2n-2m$.
In view of~\eqref{eq:eta-structure} this means to find functions~$\phi_1$, $\phi_2$ such that
\begin{align}\label{eq:phi2-G=0}
& \forall \, x \in U: \ 
\begin{bmatrix}
\phi_1'(x_1) & 0 \\
\frac{\partial}{\partial x_1} \phi_2(x_1)x_2 & \phi_2(x_1)
\end{bmatrix}
\begin{bmatrix}
0 \\
M(x_1)^{-1} B(x_1)
\end{bmatrix} = 0
\nonumber \\ 
&\iff 
 \forall\, x_1 \in U_1: \ \phi_2(x_1) M(x_1)^{-1} B(x_1) = 0.
\end{align}
In the following lemma we show the existence of functions $\phi_1, \phi_2$ satisfying the aforesaid.
\begin{Lem} \label{Lem:existence-phi1/2}
Consider the ODE~\eqref{eq:ODE} and assume~\ref{ass:hMB-regular}.
For any~$x_1^0 \in U_1 $
there exist an open neighbourhood $ U_1^0 \subseteq U_1$ of~$x_1^0$ and $ \phi_1\in \cC^1( U_1^0 \to \R^{n-m})$, 
$\phi_2\in\cC(U_1^0 \to \R^{(n-m) \times n})$ such that~\eqref{eq:Phi-Jacobian} and~\eqref{eq:phi2-G=0} hold locally on~$U_1^0$.
\end{Lem}
 
\begin{proof}
We fix $x_1^0 \in U_1$ and make use of~\cite[Lem. 4.1.5]{Sont98a} which states the following. 
Consider $W \in \cC( U_1 \to \R^{w \times n})$ with $\rk W(x_1) = w$ for all~$x_1 \in U_1$. 
Then there exist an open neighbourhood $V_1 \subseteq U_1$ of~$x_1^0$ and~$T \in \cC(V_1 \to \Gl_n(\R))$ such that
\[
\forall\, x_1 \in V_1 : \ W(x_1) T(x_1) = \begin{bmatrix} I_w & 0 \end{bmatrix}.
\]
We use this to show the existence of $\phi_1 \in \cC^1( U_1^0 \to \R^{n-m})$. 
Since by assumption~\ref{ass:hMB-regular} we have $\rk H(x_1) =  m$ for all ${x_1 \in U_1}$ there exist an open neighbourhood $V_1 \subseteq U_1$ of~$x_1^0$ and 
$T = [T_1, T_2] \in \cC(V_1 \to \Gl_n(\R))$ such that
\begin{equation*}
\forall\, x_1 \in V_1: \ 
H(x_1) 
\begin{bmatrix} 
T_1(x_1) & T_2(x_1)
\end{bmatrix}
= \begin{bmatrix}
I_{ m} & 0
\end{bmatrix},
\end{equation*}
i.e., $\im T_2(x_1) = \ker H(x_1) $ and $\rk T_2(x_1) = n-m$ for all~$x_1 \in V_1$. Let~$E = [e_{i_1}^\top,...,e_{i_{n-m}}^\top]^\top \in \R^{(n-m) \times n}$ with $e_{i_j} \in \R^{1 \times n}$ a unit row-vector for~$i_j \in \{1,...,n\}$. Then
\begin{equation*}
\begin{bmatrix}
 H(x_1) \\ E
\end{bmatrix} 
\begin{bmatrix} 
T_1(x_1) & T_2(x_1)
\end{bmatrix}
= \begin{bmatrix}
I_{ m} & 0 \\ * & E T_2(x_1)
\end{bmatrix}.
\end{equation*}
Since $\rk T_2(x_1^0) = n-m$ it is possible to choose $i_1,\dots,i_{n-m}$ such that $E T_2(x_1^0) \in \Gl_{n-m}(\R)$. 
As $T_2 \in \cC(V_1 \to \R^{n \times (n-m)})$ the mapping $x_1 \mapsto \det(ET_2(x_1))$ is continuous on~$V_1$, hence there exists an open neighbourhood $\bar V_1 \subseteq V_1$ of 
$ x_1^0$ such that $\det(ET_2(\bar x_1)) \neq 0$ for all $\bar x_1 \in \bar V_1$. 
Thus,
\begin{equation*}
\forall\, x_1 \in \bar V_1: \
\rk \begin{bmatrix}
H(x_1) \\ E
\end{bmatrix} = n.
\end{equation*}
Therefore, with
\begin{align*}
\phi_1 :  \bar V_1 &\to \R^{n-m}, \quad
x_1  \mapsto E x_1
\end{align*}
we have $\phi_1 \in \cC^1(\bar V_1 \to \R^{n-m})$ and the first condition in~\eqref{eq:Phi-Jacobian} is satisfied on~$\bar V_1$ since $\phi_1'(x_1) = E$.
\\
Now, we show the existence of~$\phi_2 \in \cC(U_1^0 \to \R^{n-m})$. 
Observe that by assumption~\ref{ass:hMB-regular} we have $ \rk B(x_1) =  m$ and $\rk M(x_1) = n$ for all~$x_1 \in U_1$. 
Therefore, again via \cite[Lem. 4.1.5]{Sont98a}, there exist an open neighbourhood $\tilde V_1 \subseteq  U_1$ of~$x_1^0$ and $T = [T_1, T_2] \in \cC(\tilde V_1 \to \Gl_n(\R))$ such that
\begin{equation*}
\forall \, x_1 \in \tilde V_1: \, [M(x_1)^{-1}  B(x_1)]^\top \begin{bmatrix} T_1(x_1) & T_2(x_1)\end{bmatrix}
= \begin{bmatrix} I_{ m} & 0 \end{bmatrix},
\end{equation*}
i.e., $\im T_2(x_1) = \ker (M(x_1)^{-1} B(x_1))^\top$ and $T_2 \in \cC(\tilde V_1 \to \R^{n \times (n-m)})$. Now, we observe for
${\phi_2(x_1) = T_2(x_1)^\top}$
\begin{equation*}
\begin{aligned}
& \begin{bmatrix} H(x_1) \\ \phi_2(x_1) \end{bmatrix}
\begin{bmatrix} M(x_1)^{-1} B(x_1) & \phi_2(x_1)^\top \end{bmatrix} 
\\ &=
\begin{bmatrix} \Gamma(x_1) & * \\ 0 & \phi_2(x_1) \phi_2(x_1)^\top  \end{bmatrix} \in \R^{n \times n}, \  x_1  \in \tilde V_1
\end{aligned}
\end{equation*}
which is invertible on~$\tilde V_1$ since by assumption~\ref{ass:hMB-regular} the high-gain matrix $\Gamma(x_1)$ is invertible on~$U_1$, and $\rk T_2(x_1) = n-m$ for all~$ x_1 \in \tilde V_1$. Hence the second condition in~\eqref{eq:Phi-Jacobian} is satisfied on~$\tilde V_1$. 
Moreover, equation~\eqref{eq:phi2-G=0} is true on~$\tilde V_1$ by construction of~$\phi_2$. 
We set $U_1^0 := \bar V_1 \cap \tilde V_1$ which completes the proof.
\end{proof}

The proof of the previous lemma justifies the structural ansatz for the internal state~$\eta$ in~\eqref{eq:eta-structure}. Hereinafter let~$U_1 = U_1^0$ with~$U_1^0$ as in Lemma~\ref{Lem:existence-phi1/2}. \\
While~$\phi_1$ can basically be chosen freely up to~\eqref{eq:Phi-Jacobian}, $\phi_2$ is uniquely determined up to an invertible left transformation. To find all possible representations, let $P: U_1\to \R^{n \times m}$ and $V: U_1\to \R^{n \times (n-m)}$ be such that
\begin{equation} \label{eq:PV-inverse}
\forall\, x_1\in U_1:\ [P(x_1), V(x_1)] \begin{bmatrix} H(x_1) \\ \phi_2(x_1) \end{bmatrix} = I_n,
\end{equation}
which exist by~\eqref{eq:Phi-Jacobian}. Then $P, V$ have pointwise full column rank, by which the pseudoinverse of~$V$ is given by 
$V^\dagger(x_1) = (V(x_1)^\top V(x_1))^{-1} V(x_1)^\top$, $x_1\in U_1$. For $x_1\in U_1$ we define
\begin{equation} \label{fcn:phi2-def}
\begin{aligned}
\phi_2(x_1) &:=  V^\dagger(x_1) \bigg(I_n - M(x_1)^{-1}  B(x_1) \Gamma(x_1)^{-1} H(x_1)  \bigg).
\end{aligned}
\end{equation}

\begin{Lem} \label{Lem:phi2}
We use the notation and assumptions from Lemma~\ref{Lem:existence-phi1/2}.
Then the function~$\phi_2: U_1 \to \R^{(n-m) \times n}$ is uniquely determined by~\eqref{eq:Phi-Jacobian} and~\eqref{eq:phi2-G=0} up to an invertible left transformation. All possible functions are given by~\eqref{fcn:phi2-def} for feasible choices of~$V$ satisfying~\eqref{eq:PV-inverse}.
\end{Lem}
\begin{proof} 
Assume that~\eqref{eq:Phi-Jacobian} and~\eqref{eq:phi2-G=0} hold, and hence we have~\eqref{eq:PV-inverse} for some corresponding~$P$ and~$V$. 
We multiply~\eqref{eq:PV-inverse} from the left by $V(x_1)^\dagger$ and subtract $V(x_1)^\dagger P(x_1) H(x_1)$ from both sides, and obtain
\begin{equation*}
\phi_2(x_1) = V^\dagger(x_1) \left(I_n - P(x_1) H(x_1)  \right),\quad x_1\in U_1.
\end{equation*}
Invoking~\eqref{eq:phi2-G=0}, we further obtain from~\eqref{eq:PV-inverse} that
\begin{align*}
P(x_1) & =  M(x_1)^{-1} B(x_1) \underset{= \Gamma(x_1)^{-1}}{\underbrace{\left( H(x_1) M(x_1)^{-1} B(x_1)  \right)^{-1}}},
\end{align*}
and hence~$P$ is uniquely determined by~$M, H, B$. Therefore,~$\phi_2$ is given as in~\eqref{fcn:phi2-def}. Furthermore, it follows from~\eqref{eq:PV-inverse} that
\[
 \begin{bmatrix}  H(x_1) \\ \phi_2(x_1) \end{bmatrix} [P(x_1), V(x_1)] = \begin{bmatrix} I_{m} &0\\0& I_{n-m} \end{bmatrix},
\]
from which we may deduce $\phi_2(x_1) V(x_1) = I_{n-m}$ and in addition $\im V(x_1) = \ker H(x_1)$. Hence, the representation of~$\phi_2$ in~\eqref{fcn:phi2-def} only depends on the choice of the basis of $\ker H(x_1)$. Now, let $\tilde V(x_1) := V(x_1) R(x_1)$, $x_1\in U_1$, for some $R: U_1\to \Gl_{n-m}$ and consider
\begin{align*}
\tilde \phi_2(x_1) &=  \tilde V^\dagger(x_1) \big(I_n -  M(x_1)^{-1} B(x_1) \Gamma(x_1)^{-1} H(x_1) \big).
\end{align*}
A short calculation shows $\tilde \phi_2(x_1) = R(x_1)^{-1} \phi_2(x_1)$ for all $x_1\in U_1$.
\end{proof}
Now, we continue deriving a representation of the internal dynamics. We choose $V \in \cC(U_1 \to \R^{n-m})$ with orthonormal columns such that~$\im V(x_1) = \ker H(x_1)$ for all~$x_1 \in U_1$.  Then via~\ref{ass:hMB-regular} and~\eqref{fcn:phi2-def} the inverse of
$
\begin{bmatrix}
H(x_1) \\ \phi_2(x_1)
\end{bmatrix}
$
is given by
\begin{equation}  \label{Cor:Inverse}
\begin{bmatrix}
H(x_1) \\ \phi_2(x_1)
\end{bmatrix}^{-1}
=
\begin{bmatrix}
M(x_1)^{-1} B(x_1) \Gamma(x_1)^{-1} & V(x_1)
\end{bmatrix}, \ x_1 \in U_1.
\end{equation}
Recall~\eqref{eq:eta-structure} and~\eqref{eq:xi-structure}, and choose~$\phi_2(x_1)$ as in~\eqref{fcn:phi2-def}. Then~\eqref{Cor:Inverse} yields
\begin{equation}
\begin{aligned} 
\begin{pmatrix} \xi_2 \\ \eta_2 \end{pmatrix} &= \begin{bmatrix} H(x_1) \\ \phi_2(x_1) \end{bmatrix} x_2 \\
\Rightarrow 
x_2 &= \begin{bmatrix} H(x_1) \\ \phi_2(x_1) \end{bmatrix}^{-1} \begin{pmatrix}  \xi_2 \\ \eta_2 \end{pmatrix}  \\
&\overset{\eqref{Cor:Inverse}}{=} M(x_1)^{-1} B(x_1) \Gamma(x_1)^{-1} \xi_2  + V(x_1)\eta_2.  \label{eq:x2}
\end{aligned}
\end{equation}
Therefore, using~\eqref{eq:ODE}, \eqref{eq:eta-structure} and~\eqref{eq:x2}, and identify~$\xi_2 = \dot y$, the dynamics of~$\eta_1$ are given by
\begin{align}
\dot{\eta}_1(t) &= \phi'_1(x_1(t)) x_2(t) \nonumber \\ 
&= \phi'_1(x_1(t)) M(x_1(t))^{-1} B(x_1(t)) \Gamma(x_1(t))^{-1}  \dot y(t) \nonumber  \\
& + \phi'_1(x_1(t)) V(x_1(t)) \eta_2(t) \label{eq:eta1-dynamik_x1_dependent},
\end{align}
where $\phi_1'(x_1)$ is of full row rank for all~$x_1 \in U_1$ by~\eqref{eq:Phi-Jacobian} but apart from that arbitrary. Again with~\eqref{eq:eta-structure} and~\eqref{eq:xi-structure}, since 
\begin{equation} \label{fcn:x1-diffeomorphism}
\begin{pmatrix} \xi_1 \\ \eta_1 \end{pmatrix} = \begin{pmatrix} h(x_1) \\ \phi_1(x_1) \end{pmatrix} =: \vp(x_1)
\end{equation}
is continuously differentiable and its Jacobian is regular on~$ U_1 \subseteq \R^n$ by~\eqref{eq:Phi-Jacobian}, via the inverse function theorem~$\vp$ defines a diffeomorphism on~$ U_1 $ and hence we have 
\begin{equation} \label{eq:x1}
x_1 = \vp^{-1}\left(\begin{pmatrix} \xi_1 \\ \eta_1 \end{pmatrix} \right). 
\end{equation}
\begin{Cor} \label{Cor:phi1}
Assume~\ref{ass:hMB-regular} holds true and that there exists $E \in \R^{(n-m) \times n}$ such that $\phi_1(x_1) := E x_1$, $x_1 \in U_1$, satisfies~\eqref{eq:Phi-Jacobian}. Further, assume there exists $H \in \R^{m \times n}$ such that~$h$ is linear with $h(x_1) = Hx_1$ for all $x_1 \in U_1$, and let $V: U_1 \to \R^{n \times (n-m)}$ be such that 
$\im V = \ker H$. Then for $\vp: U_1 \to \R^n$ defined as in~\eqref{fcn:x1-diffeomorphism} we find that
\begin{equation*}
\vp(U_1) = \begin{bmatrix} H \\ E \end{bmatrix} U_1 =: W_1
\end{equation*}
and for all~$w_1 \in W_1$ we have
\begin{equation*}
\begin{aligned}
& \varphi^{-1}(w_1) = 
\begin{bmatrix} H \\ E \end{bmatrix}^{-1} w_1  \\
&= \begin{bmatrix}
H^\top (H H^\top)^{-1} - V(EV)^{-1}EH^\top(H H^\top)^{-1} & V(EV)^{-1}
\end{bmatrix} w_1 
\end{aligned}
\end{equation*}
\end{Cor}
\begin{proof}
Clear.
\end{proof}
To combine~\eqref{fcn:x1-diffeomorphism} with~\eqref{eq:eta1-dynamik_x1_dependent} we define the following functions on~$W_1 := \vp(U_1)$ as concatenations
\begin{align*}
\phi'_{1,\vp}(\cdot) &:=  \big(\phi'_1 \circ \vp^{-1} \big)(\cdot), & M_{\vp}(\cdot)^{-1} &:= \big(M \circ \vp^{-1} \big)(\cdot)^{-1}, \\
B_{\vp}(\cdot) &:= \big(B \circ \vp^{-1} \big)(\cdot), & H_{\vp}(\cdot) &:= \big(H \circ \vp^{-1} \big)(\cdot), \\ 
\Gamma_{\vp}(\cdot)^{-1} &:= \big(\Gamma \circ \vp^{-1} \big)(\cdot)^{-1}, &V_{\vp}(\cdot) &:= \big(V \circ \vp^{-1} \big)(\cdot). &
\end{align*}
Therefore, identifying~$\xi_1 = y$ we obtain a representation of~\eqref{eq:eta1-dynamik_x1_dependent} as the first equation in~\eqref{eq:eta-dynamics}.
Now, we explore the dynamics of~$\eta_2$. 
Define $\phi'_2[x_1,x_2] := \tfrac{\partial}{\partial x_1} [\phi_2(x_1) \cdot x_2] \in \R^{(n-m) \times n }$ 
for $x = (x_1^\top,x_2^\top)^\top \in U_1 \times \R^n = U$. 
Then from~\eqref{eq:ODE} and~\eqref{eq:eta-structure} we obtain for $x \in U$
\begin{align*}
\dot \eta_2(t) &= \phi'_2[x_1(t),x_2(t)]x_2(t) \\ & + \phi_2(x_1(t)) M(x_1(t))^{-1} f(x_1(t),x_2(t)).
\end{align*}
Let $\phi_{2,\vp}(\cdot) :=  \big(\phi_2 \circ \vp^{-1} \big)(\cdot)$ on~$W_1$ and for $w \in W_1$ and $v \in \R^n$ we define
\[
\phi'_{2,\vp}[w,v] := \phi'_2 \Bigg[\vp^{-1}(w), \begin{bmatrix}H_\vp(w) \\ \phi_{2,\vp}(w) \end{bmatrix}^{-1}  v \Bigg],
\]
\[
f_{\vp}(w,v) := f \Bigg(\vp^{-1}(w), \begin{bmatrix} H_{\vp}(w) \\ \phi_{2,\vp}(w) \end{bmatrix}^{-1}  v \Bigg).
\]
Then, as the first main result of the present article, the internal dynamics of~\eqref{eq:ODE} are given in~\eqref{eq:eta-dynamics}.
\begin{figure*}
\fbox{
\parbox{\linewidth}{
\begin{equation}
\begin{aligned} \label{eq:eta-dynamics}
\dot \eta_1(t) &=
\phi'_{1,\vp} \bigg(\begin{pmatrix} y(t) \\ \eta_1(t) \end{pmatrix}\bigg) 
M_\vp\bigg(\begin{pmatrix} y(t) \\ \eta_1(t) \end{pmatrix}\bigg)^{-1} 
 B_\vp\bigg(\begin{pmatrix} y(t) \\ \eta_1(t) \end{pmatrix}\bigg)
  \Gamma_\vp\bigg(\begin{pmatrix} y(t) \\ \eta_1(t) \end{pmatrix}\bigg)^{-1} \dot y(t) \\
  &+ 
  \phi'_{1,\vp}\bigg(\begin{pmatrix} y(t) \\ \eta_1(t) \end{pmatrix}\bigg) 
  V_\vp\bigg(\begin{pmatrix} y(t) \\ \eta_1(t) \end{pmatrix}\bigg) \eta_2(t),   \\
\dot \eta_2(t) &= 
\phi'_{2,\vp} \Bigg[ \begin{pmatrix} y(t) \\ \eta_1(t) \end{pmatrix}, \begin{pmatrix} \dot y(t) \\ \eta_2(t) \end{pmatrix} \Bigg]
\Bigg( M_\vp\bigg(\begin{pmatrix} y(t) \\ \eta_1(t) \end{pmatrix}\bigg)^{-1} B_\vp\bigg(\begin{pmatrix} y(t) \\ \eta_1(t) \end{pmatrix}\bigg) 
\Gamma_\vp\bigg(\begin{pmatrix} y(t) \\ \eta_1(t) \end{pmatrix}\bigg)^{-1} \dot y(t) + V_\vp\bigg(\begin{pmatrix} y(t) \\ \eta_1(t) \end{pmatrix}\bigg) \eta_2(t) \Bigg) \\   
&+ \phi_{2,\vp}\bigg(\begin{pmatrix} y(t) \\ \eta_1(t) \end{pmatrix}\bigg) 
M_\vp\bigg(\begin{pmatrix} y(t) \\ \eta_1(t) \end{pmatrix}\bigg)^{-1} 
f_\vp \bigg(\begin{pmatrix} y(t) \\ \eta_1(t) \end{pmatrix},  \begin{pmatrix} \dot y(t) \\ \eta_2(t) \end{pmatrix} \bigg).
\end{aligned}
\end{equation}
}}
\end{figure*}
\begin{Rem}
The set of variables presented in this section to decouple the internal dynamics of~\eqref{eq:ODE} offers an alternative to the Byrnes-Isidori form as in~\eqref{eq:abstract-BIF}, whose computation often requires a lot of effort. 
The advantage of the representation of the internal dynamics in~\eqref{eq:eta-dynamics} is, that it involves only system parameters. 
Lemma~\ref{Lem:phi2} shows, that
a suitable choice of V is given by $\im V(x_1) = \ker H(x_1)$,
and solving~\eqref{eq:x1} for~$x_1$ only involves~$h(x_1)$, 
where the choice $\phi_1(x_1) = E x_1$ as in the proof of Lemma~\ref{Lem:existence-phi1/2} may simplify the computation further.
We stress, that once~$\phi_1$ is chosen, e.g. $\phi_1(x_1) = E x_1$, the computation of the internal dynamics of a system~\eqref{eq:ODE} can be carried out completely algorithmically without the need to choose further functions. 
\end{Rem}
\begin{Rem} \label{Rem:phi2-conservative}
We may obtain further structure for the internal dynamics~\eqref{eq:eta-dynamics}. 
Recall the concept of a conservative vector field. We call a vector field $J : U \to \R^n$, $U \subseteq \R^n$ open, conservative, if there exists a scalar field $j : U \to \R$ such that $j'(x) = J(x)^\top$ for all $x \in U$. Now, if there exist 
$j_i \in \cC^1(U_1 \to \R)$ such that 
$j_i'(x_1) = \phi_{2,i}(x_1)^\top$ for all $x_1 \in U_1$, $i=1,\dots,n-m$, then it is possible to choose 
$\phi_1 = (\lambda_1 j_1 ,\dots,\lambda_{n-m} j_{n-m})^\top$ for some $\lambda_i\in\R\setminus\{0\}$, $i=1, \dots, n-m$, and thus 
$\phi_1'(x_1) = \Lambda \phi_2(x_1)$ for $\Lambda = \diag(\lambda_1,\dots,\lambda_{n-m})$. Therefore, using~\eqref{eq:phi2-G=0} and Lemma~\ref{Lem:phi2} the dynamics of~$\eta_1$ in~\eqref{eq:eta-dynamics} reduce to
\begin{equation*}
\dot \eta_1(t) = \Lambda \eta_2(t).
\end{equation*}
Note, that the entries of~$\Lambda$ can be chosen at will. 
We will make use of this later in Section~\ref{Subsec:StabCond}.
\end{Rem}

\section{Stability of the internal dynamics} \label{Sec:Stability-of-ID}
In this section we derive sufficient conditions on the system parameters
of a certain class of systems with constant mass matrix,
such that the internal dynamics are bounded-input, bounded-output stable.
These conditions can be verified in advance and hence an explicit decoupling and stability analysis of the internal dynamics is not necessary; in particular, it is not necessary to derive~\eqref{eq:eta-dynamics}.

\subsection{Abstract stability result}
Before we give conditions on~$f$ in~\eqref{eq:ODE} to ensure a bounded-input, bounded-output stability of the internal dynamics~\eqref{eq:eta-dynamics} we present an abstract stability result. To this end, we make the following definitions.
\begin{Def} 
Consider a dynamical control system
\begin{equation} \label{eq:abstract-control-system}
\dot x(t) = f(x(t),u(t)), \quad x(0) = x^0 \in \R^n,
\end{equation}
where $f \in \cC(X \times \R^m \to \R^n)$, $X \subseteq \R^n$ open, and $u \in \cL^\infty(\rp \to \R^m)$. 
A local solution of~\eqref{eq:abstract-control-system} is a function $x \in \cC^1([0,\omega) \to \R^n)$, $\omega \in \rp$ such that it satisfies~\eqref{eq:abstract-control-system} on~$[0,\omega)$ with~$x(0) = x^0 \in \R^n$. If $\omega = \infty$ we call $x \in \cC^1(\rp \to \R^n)$ a global solution.
\end{Def}
\begin{Def}
We call a function $f: \R^n \to \R$ radially unbounded, if $f(x) \to \infty$ for $\| x \| \to \infty$.
\end{Def}
Now, following~\cite[Thm. 4]{LaSa60}, we may formulate the second main result of the present paper as the following theorem.
\begin{Thm} \label{Thm:Boundedness-ID}
Let~$U \subseteq \R^{n-m}$ be open and consider  
\begin{equation} \label{eq:abstract-CS}
\dot \zeta(t) = \Psi(\zeta(t), y_1(t), y_2(t)), \quad \zeta(0) = \zeta^0 \in U,
\end{equation} 
where~$\Psi \in \cC(U \times \R^{m} \times \R^m \to \R^{n-m})$. 
Assume there exist $r_1, r_2, r_3 >0 $ and a radially unbounded $\cV \in \cC^1(U \to \R)$
such that
for all $ y_1 \in B_{r_1}(0)$, $ y_2 \in B_{r_2}(0)$ and for all $\zeta \in \{ \zeta \in U \ \vline \ \|\zeta\| > r_3 \} $
we have
\begin{equation} \label{eq:Thm_Lyapunov}
\cV'(\zeta) \cdot \Psi(\zeta,y_1,y_2) \leq 0. 
\end{equation}
Then for all $y_1, y_2 \in \cL^{ \infty}(\rp \to \R^m)$ with $\| y_1 \|_{\infty} \leq r_1$ and $\|  y_2 \|_{\infty} \leq r_2$ and all global solutions $\zeta : \R_{\ge 0} \to U$ of~\eqref{eq:abstract-CS} 
there exists $\ve > 0$ such that 
\begin{equation*}
\| \zeta \|_\infty \le \max\{\|\zeta^0\|, r_3 \} + \ve .
\end{equation*}
\end{Thm} 

\begin{proof}
The following proof is inspired by the proof made in \cite[Thm. 4]{LaSa60}, and some ideas are adopted from the proof in \cite[Lem. 5.7.8]{Sont98a}.
Let $\tilde r := \max\{\|\zeta^0\|, r_3\}$ and $\nu_0 := \max \{\cV(\zeta) \ \vline \ \|\zeta\| = \tilde r \}$. 
Since~$\cV$ is radially unbounded there exists $\ve>0$ such that $\cV(\zeta) > \nu_0$ for all $\zeta \in \{ \zeta \in U \ \vline \ \| \zeta \| \ge \tilde r + \ve \}$.
Seeking a contradiction, we assume there exists $t_1$ such that $\| \zeta(t_1) \| > \tilde r + \ve $. 
Let $t_0 = \max \{ t \in [0,t_1) \ \vline \ \| \zeta(t) \| = \tilde r \} $.
Then we have $\cV(\zeta(t)) > \nu_0$ for all $t \in (t_0,t_1]$.
Since by~\eqref{eq:Thm_Lyapunov} $\cV$ is nonincreasing along solution trajectories of~\eqref{eq:abstract-CS} for all~$y_1, y_2 \in \cL^\infty(\rp \to \R^m)$ with~$\|y_1\|_\infty \leq r_1$, $\| y_2 \|_\infty \leq r_2$ and $\| \zeta \| > r_3$ we have 
$\nu_0 < \cV(\zeta(t_1)) \le \cV(\zeta(t_0)) \le \nu_0$.
A contradiction. Therefore, we conclude $\| \zeta(t) \| \le \tilde r + \ve$ for all $t \ge 0$ and hence $\| \zeta \|_\infty \le \tilde r + \ve$.
\end{proof}
\subsection{Sufficient conditions for stability} \label{Subsec:StabCond}
In this section we consider a special class of nonlinear multibody systems, namely such systems with constant mass matrix, constant input distribution and a linear system output, and  present sufficient conditions on the system parameters, i.e., on~$f$ in~\eqref{eq:ODE}, such that the internal dynamics are bounded-input, bounded-output stable. 
\begin{Def} \label{Def:BIBO}
Consider a control system~\eqref{eq:abstract-control-system} with output $y(t) = h(x(t))$, where $h \in \cC^1(\R^n \to \R^m)$.
We call a system~\eqref{eq:abstract-control-system} bounded-input, bounded-output stable, if there exist $C_1, C_2 \ge 0$ such that
\begin{equation*}
\| u \|_\infty \le C_1 \Rightarrow \| y \|_\infty \le C_2.
\end{equation*}
\end{Def}
\begin{Rem}
In the context of the internal dynamics~\eqref{eq:eta-dynamics} of system~\eqref{eq:ODE} we actually consider bounded-input, bounded-state stability. 
As we add the output $z(t) = \eta(t)$ to system~\eqref{eq:eta-dynamics} we may refer to it as bounded-input, bounded-output stable, where the output of system~\eqref{eq:ODE}, namely $y(\cdot)$, and its derivative $\dot y(\cdot)$ play the role of the input of system~\eqref{eq:eta-dynamics}.
\end{Rem}
Now, we consider a system~\eqref{eq:ODE} with constant mass matrix $M \in \Gl_n(\R)$ and constant input distribution $B \in \R^{n \times m}$. 
Since in many applications the output function~$y(\cdot)$ is linear, we assume~$h$ to be linear, i.e., $h(x_1) = H \cdot x_1$ for all $x_1 \in U_1$ with $H \in \R^{m \times n}$.
Under these assumptions we have $U_1 = \R^n$ and~\ref{ass:hMB-regular} becomes $H M^{-1} B \in \Gl_m(\R)$, i.e., $\Gamma \in \Gl_m(\R)$.
We revisit system~\eqref{eq:eta-dynamics} to obtain a simpler representation of the internal dynamics. 
First, we observe that since $H,M,B$ are constant matrices we have that
\begin{equation} \label{fn:phi2-constant}
\phi_2 = V^\dagger \big( I_n - M^{-1} B \Gamma^{-1} H \big) \in \R^{(n-m) \times n}
\end{equation}
is a constant matrix, where 
$V \in \R^{n \times (n-m)}$ such that $\im V = \ker H$
 and~$\Gamma = H M^{-1} B \in \Gl_m(\R)$.
Thence, $\phi'_2[x_1,x_2]=0$ and $U = \R^{2n}$. Furthermore, 
$\phi_2$ defines a conservative vector field and thus via Remark~\ref{Rem:phi2-conservative} we may choose $\phi_1(x_1) = \lambda \phi_2 \cdot x_1$, $x_1 \in \R^n$ and $\lambda \in \R\setminus\{0\}$. 
Moreover, since~$h$ is linear we have $V \in \R^{n \times (n-m)}$ such that $\im V = \ker H$.
Therefore, via~\eqref{eq:x2} and~\eqref{fcn:x1-diffeomorphism} we obtain
\begin{align}
x_1 &= \begin{bmatrix} H \\ \lambda \phi_2 \end{bmatrix}^{-1} \begin{pmatrix} \xi_1 \\ \eta_1 \end{pmatrix} = M^{-1} B \Gamma^{-1}  \xi_1 + \lambda^{-1} V \eta_1
\label{eq:x1-const} ,\\
x_2 &= \begin{bmatrix} H \\ \ \phi_2 \ \end{bmatrix}^{-1} \begin{pmatrix} \xi_2 \\ \eta_2 \end{pmatrix} = M^{-1} B \Gamma^{-1} \xi_2 +  V \eta_2
\label{eq:x2-const}.
\end{align}
We define $\cM := M^{-1} B \Gamma^{-1} \in \R^{n \times m}$ and $\Theta := \phi_2 M^{-1}$. Then, combining the aforementioned observations, equations~\eqref{eq:eta-dynamics} for the internal dynamics simplify to
\begin{align*}
\dot{\eta_1}(t) &= \lambda \eta_2(t) , \\
\dot{\eta_2}(t) &= \Theta f \big( \cM  y(t) + \lambda^{-1} V \eta_1(t), \, \cM \dot y(t) +  V \eta_2(t)\big),
\end{align*}
where we identified $\xi_1 = y$, $\xi_2 = \dot y$, and the arguments of $f(x_1,x_2)$ have been substituted via~\eqref{eq:x1-const} and~\eqref{eq:x2-const}, resp.
Now, we assume~$f$ has the structure
\begin{equation*}
f(x_1,x_2) = -K(x_1) - D_2(x_2) - C(x_1,x_2) x_2,
\end{equation*}
where~$K \in \cC(\R^n \to \R^n)$ may be considered as a nonlinear restoring force, 
$D \in \cC( \R^n \to \R^n)$ for example mimics a nonlinear damping
or friction
and $C \in \cC( \R^n \times \R^n \to \R^{n \times n})$ may take the role of a nonlinear distribution for a position dependent damping or mimic a Coriolis force. 
With this we revisit system~\eqref{eq:ODE} and obtain the following control system
\begin{small}
\begin{align} \label{eq:system-ODE-with-f-structure}
\dot x(t) &= \begin{pmatrix} x_2(t) \\ M^{-1} \big( -K(x_1(t)) - D(x_2(t)) - C(x_1(t),x_2(t))x_2(t) \big) \end{pmatrix} \nonumber \\
&+ \begin{bmatrix} 0 \\ M^{-1} B \end{bmatrix} u(t) , \nonumber \\
y(t) &= H x_1(t).
\end{align}
\end{small}
We define the vector field~$\cF$ in~\eqref{def:cF}.
\begin{figure*}[h!tb]
\begin{align} \label{def:cF}
\cF : \R^{2(n-m)} \times \R^{2m} & \to \R^{2(n-m)} \\
(z_1, z_2, v_1, v_2) & \mapsto 
\begin{pmatrix} 
\lambda z_2 \\
 \Theta \left(- K(\cM v_1 + \lambda^{-1} V z_1) - D(\cM v_2 + V z_2) - C(\cM v_1 + \lambda^{-1} V z_1, \cM v_2 + V z_2)(\cM v_2 + V z_2) \right)
 \end{pmatrix} \nonumber
\end{align}
\end{figure*}
Then the internal dynamics of~\eqref{eq:system-ODE-with-f-structure} are given via
\begin{equation} \label{eq:ID_constant_matrices}
\dot \eta(t) = \cF\big(\eta_1(t), \eta_2(t), y(t), \dot y(t)\big).
\end{equation}
Henceforth let~$V$ have orthonormal columns.
For $i=1,2$ we assume, that there exist $z_i^+ > 0$ such that
for all $z_i \in Z_i := \{ z \in \R^{n-m} \ \vline \ \|z\| > z_i^+ \}$ and all $v,w \in \R^n$ 
the functions~$K$, $ D$ and~$C$ satisfy the following conditions 
\begin{subequations}
\begin{align}
\text{There ex. a rad. unbounded } & V_K \in \cC^1( Z_1 \to \R) \nonumber \\ 
 \text{such that} \ V_K'(z_1) &= \left(\Theta K(V z_1)\right)^\top , \label{eq:K-cons} \\
\|K(V z_1)-K( V z_1+w)\| &\leq g_1(w), \label{eq:K-upper} \\
z_1^\top \Theta K(V z_1) &\geq \kappa \|z_1\|^2 , \label{eq:K-lower} \\
\|D(V z_2) - D(V z_2+w)\| &\leq g_2(w),  \label{eq:D2-upper} \\
z_2^\top \Theta D(V z_2) &\geq \delta \|z_2\|^2, \label{eq:D2-lower} \\
\|D( V z_2)\| &\leq d \|z_2\|, \label{eq:D2-Lipschitz} \\
\|(C(V z_1,w)- C(V z_1+v,& w))  w\|  \label{eq:C-upper1} \\
&\le g_3(v) a_3(w) , \nonumber \\
z_1^\top \Theta C(V z_1,w) w &\ge \|z_1\|^2 b_3(w),  \label{eq:C-lower1} \\
\|(C(v, V z_2) -  C (v,V z_2  + & w)) (V z_2 + w)\| \label{eq:C-upper2}   \\
& \le  g_4(w)  a_4(v)  \|z_2 \| , \nonumber \\ 
z_2^\top \Theta C(v,Vz_2) V z_2 &\ge b_4(v) \|z_2\| \|Vz_2\|^2 , \label{eq:C-lower2} \\
\| C(v, V z_2) w \| &\le  a_4(v) \|z_2\| \|w\|, \label{eq:C-Lipschitz} 
\end{align}
\end{subequations}
for suitable functions $a_j,b_j,g_i \in \cC(\R^n \to \rp)$, $i = 1,\ldots,4$, $j=3,4$,
with $ a_i(x) \le b_i(x)$ for all $x \in \R^n$, and $\kappa, \delta, d > 0$. 
From this it is clear that conditions~\eqref{eq:K-cons}--\eqref{eq:C-Lipschitz} mean that the acting forces are assumed to be basically linear in a certain region, far away from the origin. Hence these are merely weak assumptions.\\
We set $\tau := \|\Theta\|$ and for some~$r_1, r_2 \geq 0$, $q >0$ we define the following constants
\begin{equation} \label{def:constants}
\begin{aligned}
\tilde K &:= \max_{z \in B_{r_1}(0)} g_1(\cM z), \  && \tilde D := \max_{z \in B_{r_2}(0)} g_2(\cM z), \\
\tilde C_3 &:= \max_{z \in B_{r_1}(0)} g_3(\cM z), \ && \tilde C_4 := \max_{z \in B_{r_2}(0)} g_4(\cM z), \\
\varepsilon_1 &:= q(\tfrac{\kappa}{\lambda} - \tfrac{\tau^2}{2}), \  && \varepsilon_2 := \delta - q( \tfrac{d^2}{2} + \lambda), \\
E_1 &:= q \tau (\tilde K +  d \|\cM\| r_2), \ && E_2 := \tau (\tilde K + \tilde D  ),
\end{aligned}
\end{equation}
which are all nonnegative with the feasible choice $0 < \lambda < 2\kappa/\tau^2$ and $0 < q < 2 \delta/( d^2 + 2\lambda )$.
Further, we define
\begin{equation*} 
\begin{aligned}
\gamma_1 &:= \lambda \tau \tilde C_3 \ge 0, \ \gamma_2:= \tau (\tilde C_4 +  \mu r_2 ) \ge 0 , \\
\tilde Z_i & := \setdef{ z \in \R^{n-m} }{ \|z\| > \max \left\{ \frac{E_i}{\ve_i}, \gamma_i \right\} }, \ i=1,2.
\end{aligned}
\end{equation*}
As the third main result we present an explicit Lyapunov function for system~\eqref{eq:ID_constant_matrices} in the following theorem.
\begin{Thm} \label{Thm:V-is-a-Lyapunov-function}
Consider system~\eqref{eq:ID_constant_matrices} and fix $r_1,r_2 \ge 0$, $0 < \lambda < 2\kappa/\tau^2 $ and $0 < q < 2 \delta/( d^2 + 2\lambda )$.
Assume conditions~\eqref{eq:K-cons}--\eqref{eq:C-Lipschitz} hold true for all~$z_i \in Z_i$, $i=1,2$. 
Then for~$\cV: \R^{n-m} \times \R^{n-m} \to \R$ defined by
\begin{equation} \label{fcn:V}
\cV(\eta_1,\eta_2) = \frac{1}{2}\|\eta_2\|^2 + q \eta_1^\top \eta_2 +  V_K( \lambda^{-1} \eta_1) 
\end{equation}
the Lie derivative along the vector field~$\cF$ from~\eqref{eq:ID_constant_matrices} is nonincreasing for all $y_1 \in B_{r_1}(0)$, $y_2 \in B_{r_2}(0)$ and all $\eta_i \in Z_i \cap \tilde Z_i$, $i=1,2$, i.e.,
\begin{equation} \label{eq:Lyapunov_leq_0}
\cV'(\eta_1, \eta_2) \cdot \cF(\eta_1, \eta_2, y_1,  y_2) \le 0 .
\end{equation}
\end{Thm}
The proof is relegated to~\ref{Proof:Thm-V-is-a-Lyapunov-function}.
\begin{Rem}
The sets~$Z_i$ in Theorem~\ref{Thm:V-is-a-Lyapunov-function} are determined by the system parameters only and hence conditions 
\eqref{eq:K-cons}--\eqref{eq:C-Lipschitz} can be verified without decoupling the internal dynamics.
\end{Rem}
Finally, we combine Theorem~\ref{Thm:Boundedness-ID} and Theorem~\ref{Thm:V-is-a-Lyapunov-function} to obtain a stability result for the internal dynamics~\eqref{eq:ID_constant_matrices} of system~\eqref{eq:system-ODE-with-f-structure}.
\begin{Thm} \label{Thm:1+2}
Consider system~\eqref{eq:system-ODE-with-f-structure} and use the assumptions from Theorem~\ref{Thm:V-is-a-Lyapunov-function}. Then the internal dynamics~\eqref{eq:ID_constant_matrices} of system~\eqref{eq:system-ODE-with-f-structure} are bounded-input, bounded-output stable.
\end{Thm}
\begin{proof}
Clear.
\end{proof}
\begin{Rem}
Note, that since $\eta_i \notin Z_i \cap \tilde Z_i$ means $\| \eta_i \| \le \max\{z_i^+,  E_i/\ve_i, \gamma_i^+ \} $, it is clear that the choice of~$\lambda$ and~$q$ in Theorem~\ref{Thm:V-is-a-Lyapunov-function} does not effect the stability statement in Theorem~\ref{Thm:1+2} but only determines 
the region where~\eqref{eq:Lyapunov_leq_0} is true.
\end{Rem}
\begin{Rem}
Since we require quite restrictive conditions for systems of the form~\eqref{eq:system-ODE-with-f-structure} to provide sufficient condition on the system parameters to ensure bounded-input, bounded-output stability, we stress that a stability analysis of a given system can be performed analysing equations~\eqref{eq:eta-dynamics}. In Example~\ref{Ex:Roboterarm} we will decouple the internal dynamics of a nonlinear system, which does not belong to the class of systems~\eqref{eq:system-ODE-with-f-structure}. The resulting equations~\eqref{eq:Roboter-ID} are open to stability analysis, cf.~\cite[Sec. 4.3]{BergLanz20}.
\end{Rem}
\section{Examples} \label{Sec:Example}
In this section we present two examples to illustrate the results of the present article. First, we consider a robotic manipulator arm with non-constant mass matrix and nonlinear equations of motion, in order to demonstrate the decoupling of the internal dynamics.
Second, we consider an extension of the standard mass on a car system to illustrate how to check bounded-input, bounded-output stability of the internal dynamics without the need to decouple the internal dynamics explicitly.
\subsection{Example: Robotic manipulator} \label{Ex:Roboterarm}
In order to demonstrate the decoupling of the internal dynamics using the results from Secion~\ref{Sec:Representation-ID}, we consider a robotic manipulator arm as in~\cite[Sec. 4.2]{SeifBlaj13}. The rotational manipulator arm consists of two links with homogeneous mass distribution and mass~$m$ with length~$l$. A passive joint consisting of a linear spring-damper combination couples the two links to each other. Passive in this context means, that there is no input force at this point. We stress, that the linearity of the passive joint does not result in linear equations of motion, 
see~\eqref{eq:Roboterarm}.
As an output we measure the position of point~$S$ on the second link. Using a body fixed coordinate system the point~$S$ on the passive link is described by ${0 \leq s \leq l}$. 
The situation is depicted in Figure~\ref{Fig:Roboterarm}.
\begin{figure}[h!] 
\begin{center}
\includegraphics[width=0.75\linewidth,angle=0]{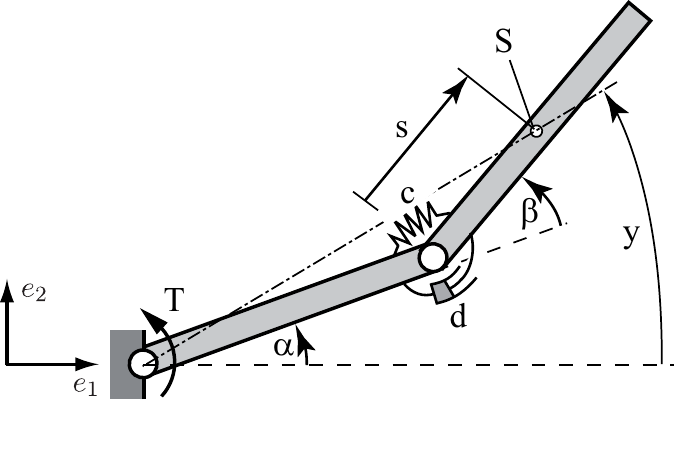}
\caption{Rotational mapilator arm consisting of two links and a passive joint. The figure is taken from from~\cite{SeifBlaj13}.}
\label{Fig:Roboterarm}
\end{center}
\end{figure}
\ \\
We present the manipulator's equations of motion.
As we will see in a second, it is reasonable to consider the dynamics of the manipulator 
for $\beta \in \cB :=  \left\{ \beta \in \R \ \vline \ \cos(\beta) \neq \tfrac{2 l}{3 s}  \right\}$.
Henceforth, we assume~$\beta \in \cB$ and perform the computations. 
We define $U_\beta := \R \times \cB \times \R^2$, and set
\begin{align*}
 M: \cB & \to \R^{2 \times 2} ,\\
x_2 &\mapsto 
l^2 m \begin{bmatrix} 
\tfrac{5}{3} + \cos(x_2) & \tfrac{1}{3} + \tfrac{1}{2}\cos(x_2) \\ \tfrac{1}{3} + \tfrac{1}{2}\cos(x_2) & \frac{1}{3}
\end{bmatrix}, \\
 f_1 : U_\beta &\to \R \\
(x_1 , x_2 , x_3 , x_4 )^\top &\mapsto \tfrac{1}{2} l^2 m x_4 (2 x_3 + x_4) \sin(x_2), \\
f_2 : U_\beta &\to \R \\
(x_1 , x_2 , x_3 , x_4 )^\top &\mapsto -c x_2 - d x_4 - \frac{1}{2} l^2 m x_3^2 \sin(x_2) ,
\end{align*}
and obtain the equations of motion
\begin{align}
\label{eq:Roboterarm} 
\begin{pmatrix}
\ddot{\alpha}(t) \\ \ddot{\beta}(t)
\end{pmatrix} &= M(\beta(t))^{-1} \begin{pmatrix} f_1(\alpha(t),\beta(t),\dot{\alpha}(t), \dot{\beta}(t)) \\ f_2(\alpha(t),\beta(t),\dot{\alpha}(t), \dot{\beta}(t)) \end{pmatrix} \nonumber \\
 &+ M(\beta(t))^{-1} \begin{bmatrix} 1 \\ 0 \end{bmatrix} u(t).
\end{align}
For later use, we compute the inverse of the mass matrix
\begin{small}
\begin{align*}
M(\beta)^{-1} = \frac{36 (l^2m)^{-2}}{16-9\cos(\beta)^2} 
\begin{bmatrix} 
 \tfrac{1}{3} & -\tfrac{1}{3} - \tfrac{1}{2}\cos(\beta) \\ -\tfrac{1}{3} - \tfrac{1}{2}\cos(\beta) & \tfrac{5}{3} + \cos(\beta)
\end{bmatrix}.
\end{align*}
\end{small}
Following the considerations in~\cite{SeifBlaj13}, we take the auxiliary angle 
\begin{equation*}
y(t) = h(\alpha(t),\beta(t)) = \alpha(t) + \frac{s}{s+l} \beta(t)
\end{equation*} 
as output, which approximates the position~$S$ on the passive link for small angles~$\alpha$ and~$\beta$.
In order to represent system~\eqref{eq:Roboterarm} in the form of~\eqref{eq:ODE} we set
\begin{equation*}
\begin{aligned}
F : U_\beta  &\to \R^4 \\
x=(x_1,\ldots,x_4)^\top & \mapsto \diag(I_2, M(x_2)^{-1}) \begin{pmatrix} x_3 \\ x_4 \\ f_1(x) \\ f_2(x) \end{pmatrix}, \\
G : U_\beta &\to \R^{4 \times 1} \\
x=(x_1,\ldots,x_4)^\top & \mapsto \diag(I_2, M(x_2)^{-1}) \begin{pmatrix} 0 \\ 0 \\ 1 \\ 0 \end{pmatrix}, \\
\tilde h : U_\beta &\to \R \\
x=(x_1,\ldots,x_4)^\top & \mapsto \begin{bmatrix} 1 & \frac{s}{s+l} &0&0 \end{bmatrix} x
\end{aligned}
\end{equation*}
and obtain
\begin{equation*}
\begin{aligned}
\dot x(t) &= F(x(t)) + G(x(t)) u(t), \\
y(t) &= \tilde h(x(t)).
\end{aligned}
\end{equation*}
Now, a short calculation gives the high gain matrix
\begin{equation*}
\begin{aligned}
\Gamma: U_\beta &\to \R \\
x & \mapsto h'(x_1,x_2) M(x_2)^{-1} B =  \begin{bmatrix} 1 & \frac{s}{s+l} \end{bmatrix} M(x_2)^{-1} \begin{bmatrix} 1 \\ 0 \end{bmatrix} \\
& = \frac{36 (l^2m)^{-2}}{16-9\cos(x_2)^2} \left( \frac{1}{3} - \frac{s}{s+l}\left(\frac{1}{3} + \frac{1}{2}\cos(x_2) \right) \right) ,
\end{aligned}
\end{equation*}
which is invertible on~$U_\beta$. Since $\Gamma \neq 0$ for~$x \in U_\beta$, Assumption~\ref{ass:hMB-regular} is satisfied and hence Lemma~\ref{Lem:r-2} yields that system~\eqref{eq:Roboterarm} has relative degree~$r = 2$ on~$U_\beta$.
Note, that for $\beta \in \left\{ \beta \in \R \ \vline \ \cos(\beta) = \tfrac{2 l}{3 s}  \right\}$ 
we have $\Gamma = 0$, from which it is clear why to consider the dynamics for $\beta \in \cB$ only.
Now, we decouple the internal dynamics of system~\eqref{eq:Roboterarm}. We start with the computation of~$\phi_1,\phi_2$ satisfying~\eqref{eq:Phi-Jacobian},\eqref{eq:phi2-G=0}.
First, we calculate~$V: \R \times \cB \to \R^2$ such that $\im V(x_1,x_2) = \ker h'(x_1,x_2)$ 
and obtain $V = \begin{smallbmatrix} - \tfrac{s}{s+l} & 1 \end{smallbmatrix}^\top$.
According to Lemma~\ref{Lem:existence-phi1/2} we may choose for~$\phi_1$ in~\eqref{eq:eta-structure}
$
\phi_1(x_1,x_2) = E \begin{smallpmatrix} x_1 \\ x_2 \end{smallpmatrix} =: [0,1] \begin{smallpmatrix} x_1 \\ x_2 \end{smallpmatrix} ,
$
and using Corollary~\ref{Cor:phi1} we obtain the expression as in~\eqref{eq:x1}
\begin{equation} \label{eq:x_1-example}
\begin{pmatrix} x_1 \\ x_2 \end{pmatrix} = \varphi^{-1}\left(\begin{pmatrix} \xi_1 \\ \eta_1 \end{pmatrix} \right) = \begin{bmatrix} 1 & -\frac{s}{s+l} \\ 0 & 1 \end{bmatrix}\begin{pmatrix} \xi_1 \\ \eta_1 \end{pmatrix} = \begin{pmatrix} \xi_1 - \frac{s}{s+l} \eta_1 \\ \eta_1 \end{pmatrix}.
\end{equation}
Next, we algorithmically compute~$\phi_2$ in~\eqref{eq:eta-structure} according to Lemma~\ref{Lem:phi2}, namely
\begin{equation*}
\begin{aligned}
 \phi_2(x_1,x_2) &= \tilde V(x_2)^\dagger \left(I_2 - M(x_2)^{-1}B \Gamma(x)^{-1} h'(x_1,x_2) \right) \\
 &= \tilde V(x_2)^\dagger \left( I_2 - M(x_2)^{-1} \begin{bmatrix} 1 \\ 0 \end{bmatrix}  \Gamma(x)^{-1} \begin{bmatrix} 1 & \frac{s}{s+l} \end{bmatrix}  \right) \\
 &= \begin{bmatrix} \frac{1}{3} +\frac{1}{2}\cos(x_2) & \frac{1}{3} \end{bmatrix}  ,
\end{aligned}
 \end{equation*} 
where $\tilde V^\dagger$ denotes the pseudoinverse of~$\tilde V$, and we chose $\tilde V(x_2) = V R(x_2)$ 
with $R(x_2) = -\tfrac{3 s \cos(x_2) - 2l}{6(s+l)}$ as a left transformation. We stress, that~$R: \cB \to \R$ is invertible on~$ \cB$ and we use it only for the sake of better legibility. 
Thence, we obtain an expression as in~\eqref{eq:x2}
\begin{equation} \label{eq:x_2-example}
\begin{aligned}
\begin{pmatrix} x_3 \\ x_4 \end{pmatrix} &= M(x_2)^{-1} B \Gamma(x)^{-1} \xi_2 + V R(x_2)^{-1} \eta_2.
\end{aligned}
\end{equation}
Now, we substitute the expressions from~\eqref{eq:x_1-example},\eqref{eq:x_2-example} into the respective functions and obtain via~\eqref{eq:eta-dynamics} the internal dynamics of system~\eqref{eq:Roboterarm}. The internal dynamics of~\eqref{eq:Roboterarm} are given in~\eqref{eq:Roboter-ID}.
\begin{figure*}
\begin{equation} \label{eq:Roboter-ID}
\begin{aligned}
\dot \eta_1(t) = & \, -\frac{(s+l)\left( (3\cos(\eta_1(t)) + 2) \dot y(t) - 6 \eta_2(t) \right)}{2l - 3s\cos(\eta_1(t))}  \\
\ \\
\dot \eta_2(t) = &\, \frac{\sin\left(\eta_{1}(t)\right)}{2} 
\frac{(3\cos(\eta_1(t)) + 2) \dot y(t) - 6 \eta_2(t)}{2l-3s \cos(\eta_{1}(t))}
\cdot \frac{ 2 (l+s)  \left( (l + s)\dot y(t) -3 s \eta_{2}(t)\right) }{2l-3s \cos(\eta_{1}(t)} \\
&-\frac{ 2 \sin(\eta_1(t))  \left( (l + s)\dot y(t) -3 s \eta_{2}(t)\right)^2 }{ (2l-3s \cos(\eta_{1}(t))^2 } 
- \frac{c}{l^2 m} \eta_1(t) 
- \frac{d}{l^2 m} \frac{(s+l)\left( (3\cos(\eta_1(t)) + 2) \dot y(t) - 6 \eta_2(t) \right)}{2l - 3s\cos(\eta_1(t))}
\end{aligned}
\end{equation}
\end{figure*}
\begin{Rem}
We highlight, that the computation of the internal dynamics is completely determined by system parameters 
and the application of Lemma~\ref{Lem:existence-phi1/2}, 
Lemma~\ref{Lem:phi2} and Corollary~\ref{Cor:phi1}, and equations~\eqref{fcn:phi2-def},\eqref{eq:x2} and~\eqref{eq:x1}.
Once chosen~$\phi_1$, e.g. $\phi_1(x_1) = E x_1$ as in the proof of Lemma~\ref{Lem:existence-phi1/2}, the decoupling of the internal dynamics can be performed computationally by an algorithm.
\end{Rem}

\subsection{Example: Mass on car}
We illustrate the stability result presented in Section~\ref{Sec:Stability-of-ID}. 
To this end, we consider the mass on a car system presented in~\cite[Sec. 4.2]{OttoSeif18b}; this is an extension of the classical mass on a car system under consideration in~\cite[Sec. 4.1]{SeifBlaj13}.
The system consists of two cars with mass~$m_1$ and~$m_2$, resp. The two cars are coupled via a spring-damper combination with characteristics~$K_2$ and~$D_2$, resp. On the second car a ramp with constant angle~$ 0<\alpha<\pi/2$ is mounted, on which a third mass~$m_3$ is lying and is coupled to the second car via a spring with characteristic $K_3$, and a damping with characteristic $D_3$. The first car is driven by a force $u_1$, and the second car individually is driven by a force~$u_2$. 
We measure the horizontal position of the first and the second car.
The situation is depicted in Figure~\ref{Fig:Wagen}.
\begin{figure}[H]
\begin{center}
\includegraphics[scale=0.30]{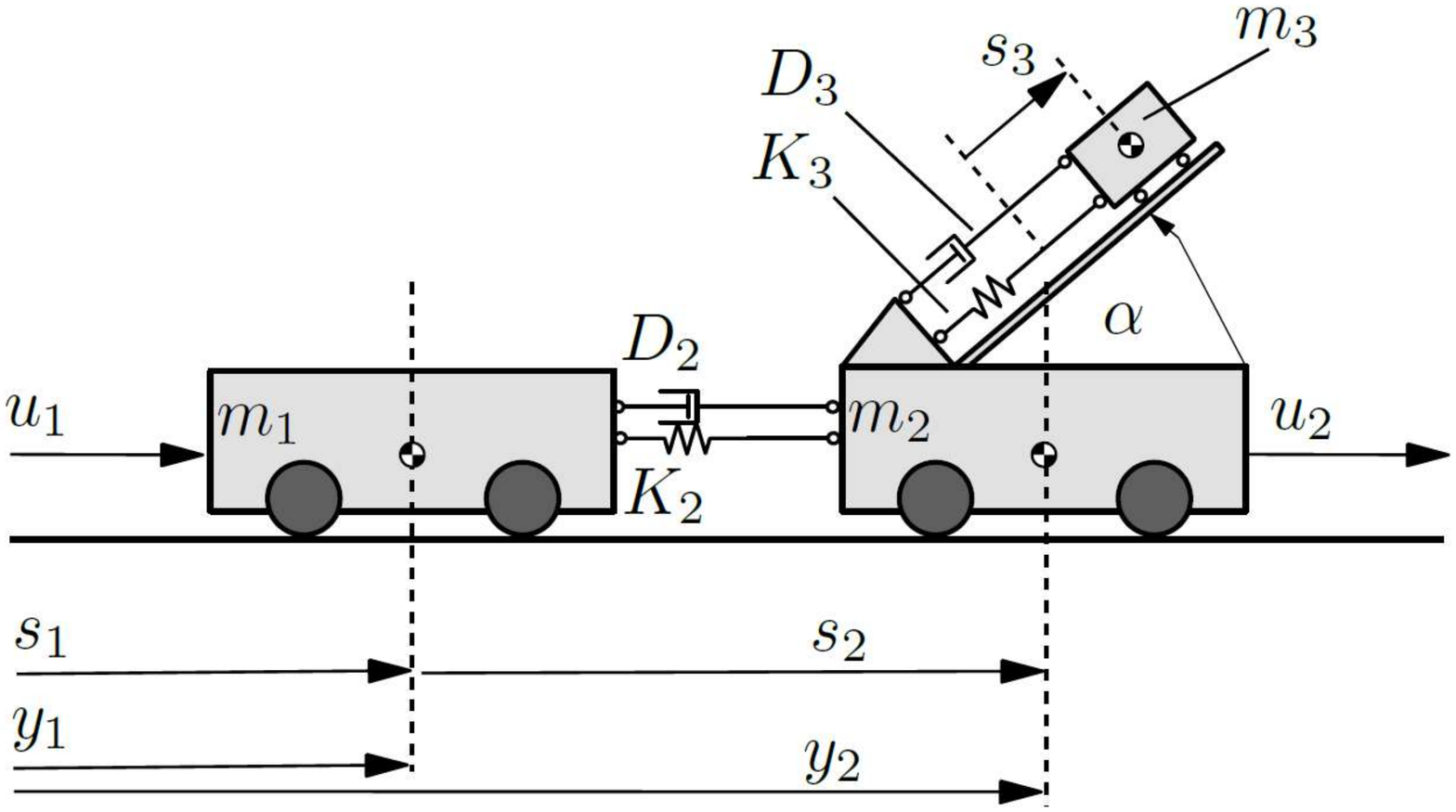}
\caption{Extended mass on a car system. The original figure is taken from~\cite{OttoSeif18b} and edited for the purpose of the present article.}
\label{Fig:Wagen}
\end{center}
\end{figure}
For convenience we assume the constant force on $m_3$ due to gravity, namely $m_3 g \sin(\alpha)$, where~$g$ is the gravitational constant, to be compensated via a linear coordinate transformation, such that $K_3(0) = 0$.
Then, according to~\cite[Sec. 4.2]{OttoSeif18b} with $s:=(s_1,s_2,s_3)^\top \in \R^3$ the equations of motion for that system are given by
\begin{equation} \label{Ex:Wagen}
\begin{aligned}
&\underbrace{
\begin{bmatrix} m_1 + m_2 + m_3 & m_2 + m_3 &  m_3 \cos(\alpha) \\ m_2 +m_3 & m_2 + m_3 & m_3 \cos(\alpha) \\ m_3 \cos(\alpha) & m_3 \cos(\alpha) & m_3 \end{bmatrix}}_{=:  M}
 \begin{pmatrix} \ddot s_1(t) \\ \ddot s_2(t) \\ \ddot s_3(t) \end{pmatrix} \\
  &  = \begin{pmatrix} 0 \\ -K_2(s(t))  - D_2(\dot{s}(t)) \\ -K_3(s(t))  - D_3(\dot{s}(t)) \end{pmatrix} 
+ \underbrace{\begin{bmatrix} 1 & 0\\ 0 & 1 \\ 0& 0 \end{bmatrix}}_{=: B} \begin{pmatrix} u_1(t) \\ u_2(t) \end{pmatrix} , \\
 y(t) & = \begin{pmatrix} y_1(t) \\ y_2(t) \end{pmatrix} = \begin{pmatrix} s_1(t) \\ s_1(t) + s_2(t) \end{pmatrix} = \underbrace{\begin{bmatrix} 1& 0 & 0 \\ 1 & 1 & 0 \end{bmatrix}}_{=: H} \begin{pmatrix} s_1(t) \\ s_2(t) \\ s_3(t) \end{pmatrix} .
\end{aligned}
\end{equation}
In this particular example we have $n=3$ and $m=2$, hence we are in the situation of multi-input, multi-output. 
Note, that the input and the output are not colocated, i.e., $H \neq B^\top$. 
According to the equations given in~\cite{OttoSeif18b} we assume $K_2(s) = k s_2$ and $D_2(\dot s_2) = d \dot s_2$, where $k,d > 0$. In order to include nonlinear terms, we assume that $K_3$ and $D_3$ have the following characteristics, where~$\sigma$ denotes the sign function
\begin{align*}
K_3 : \R^3 &\to \R , \ q  \mapsto \begin{cases}
 \sigma (q_3) \sqrt{|q_3|}  & |q_3| \le 1 \\
  2 q_3 - \sigma (q_3)  & |q_3| > 1
\end{cases}
\end{align*}
and
\begin{align*}
D_3 : \R^3 &\to \R , \ v  \mapsto \begin{cases}
 \sigma(v_3) v_3^2  & |v_3| \le 1 \\
 2 v_3 - \sigma(v_3) & |v_3| > 1.
\end{cases}
\end{align*}
Note, that $K_3 \in \cC(\R^3 \to \R)$ and $D_3 \in \cC(\R^3 \to \R)$.
The schematic shapes of $K_3(\cdot)$ and $D_3(\cdot)$ are depicted in Figure~\ref{Fig:Shape-K-D}.
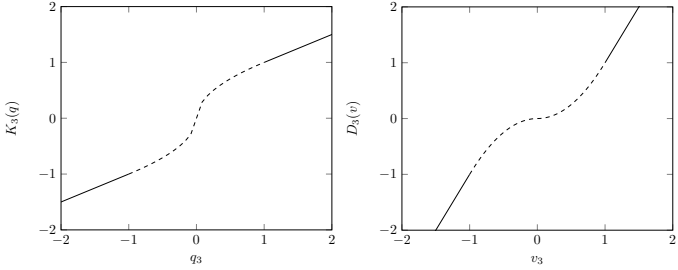
\begin{figure}[h!]
\begin{minipage}{0.48\linewidth}
\begin{tikzpicture}[scale=0.52]
\begin{axis}[%
 xmin=-2,
 xmax=2,
 ymin=-2,
 ymax=2,
ytick={2,1,0,-1,-2},
xtick={-2,-1,0,1,2},
 xlabel={$ q_3$} ,
 ylabel={$ K_3(q)$},
 ]%
\addplot[domain=1:3]{0.5*x + sign(x)*0.5};
\addplot[domain=-3:-1]{0.5*x + sign(x)* 0.5};
\addplot[dashed, domain=-1:1]{sign(x)*sqrt(abs(x))};
\end{axis}
\end{tikzpicture}
\end{minipage}
\
\begin{minipage}{0.48\linewidth}
\begin{tikzpicture}[scale=0.52]
\begin{axis}[%
 xmin=-2,
 xmax=2,
 ymin=-2,
 ymax=2,
ytick={2,1,0,-1,-2},
xtick={-2,-1,0,1,2},
 xlabel={ $ v_3$} ,
 ylabel={ $ D_3(v)$},
 ]%
\addplot[dashed, domain=-1:0]{-x^2};
\addplot[dashed, domain=0:1]{x^2};
\addplot[ domain=1:3]{2*x-1};
\addplot[ domain=-3:-1]{2*x+1};
\end{axis}
\end{tikzpicture}
\end{minipage}
\caption{Schematic shape of $K_3(\cdot)$ and $D_3(\cdot)$, resp. Solid lines on $Z_i$, dashed lines on $\R \setminus Z_i$, $i=1,2$.}
\label{Fig:Shape-K-D}
\end{figure}
\ \\
We set $x_1 :=(s_1,s_2,s_3)^\top$, $x_2 := \dot x_1$, $K(\cdot) := (0,K_2(\cdot),K_3(\cdot))^\top$, $D(\cdot) := (0,D_2(\cdot),D_3(\cdot))^\top$, whereby $K,D \in \cC(\R^3 \to \R^3)$, resp. Further, let $\tilde H := [H,0]$ 
and $\tilde M := \diag(I_3, M)$. Then system~\eqref{Ex:Wagen} reads 
\begin{align} \label{eq:Example-ode}
\dot x(t) &= \tilde M^{-1}  \begin{pmatrix} x_2(t) \\ - K(x_1(t)) - D(x_2(t)) \end{pmatrix}  
+ \tilde M^{-1} \begin{bmatrix} 0 \\ B \end{bmatrix}  \begin{pmatrix} u_1(t) \\ u_2(t) \end{pmatrix}, \nonumber
\\
y(t) &= \tilde H x(t),
\end{align}
and~\eqref{eq:Example-ode} is of the form~\eqref{eq:system-ODE-with-f-structure}.
We set $\mu := m_2 + m_3 \sin(\alpha)^2$
and calculate
\begin{align*}
M^{-1} = \begin{bmatrix} \frac{1}{m_1} & - \frac{1}{m_1} & 0\\ -\frac{1}{m_1} & \frac{m_1 + \mu}{m_1 \mu} & -\frac{\cos(\alpha)}{\mu} \\ 
0 & -\frac{\cos(\alpha)}{\mu} & \frac{m_2 + m_3}{\mu m_3} \end{bmatrix},  
\end{align*}
and 
\begin{equation*}
\begin{aligned}
\Gamma = H M^{-1} B = \begin{bmatrix} \frac{1}{m_1} & -\frac{1}{m_1} \\ 0 & \frac{1}{\mu} \end{bmatrix} \in \Gl_2(\R), \ 
\Gamma^{-1} = \begin{bmatrix} m_1 & \mu \\ 0 & \mu \end{bmatrix}.
\end{aligned}
\end{equation*}
Therefore, assumption~\ref{ass:hMB-regular} is satisfied and thus, using Lemma~\ref{Lem:r-2}, system~\eqref{eq:Example-ode} has relative degree $r=2$ on~$\R^{6}$.
We calculate~$V$ such that $\im V = \ker H$ and $ V^\top V = I_{2-1}$, thus $V = (0,0,1)^\top$.
Then, according to~\eqref{fcn:phi2-def} $\phi_2$ is given via
\begin{equation*}
\begin{aligned}
\phi_2 &= \begin{pmatrix} 0&0&1 \end{pmatrix} (I_3 - M^{-1}B \Gamma^{-1}H) \\
&= \begin{bmatrix} \cos(\alpha) &  \cos(\alpha)  & 1 \end{bmatrix}
\end{aligned}
\end{equation*}
and thus
\begin{equation*}
\begin{aligned}
\Theta &:= \phi_2 M^{-1} =  \begin{bmatrix} 0 & 0 & \frac{1}{m_3} \end{bmatrix} .
\end{aligned}
\end{equation*}
Obviously, we have $Z_i = \{ z \in \R \ \vline \ |z|> 1 \}$, $i=1,2$.
Now, we validate conditions~\eqref{eq:K-cons}--\eqref{eq:D2-Lipschitz} step by step.
Consider $V_K : Z_1 \to \R$ defined by
\begin{align*}
V_K: Z_1 &\to \R, \
z_1 \mapsto \frac{1}{m_3} (z_1^2 -  |z_1|) ,
\end{align*}
which is radially unbounded. Note, that $Z_1 = \R \setminus [-1,1]$ and hence $V_K \in \cC^1(Z_1 \to \R)$. Then for $z_1 \in Z_1$ the derivative of~$V_K$ is given by 
\begin{align*}
V_K'(z_1) &= \frac{1}{m_3} ( 2 z_1 -  \sigma(z_1))  \\
&= \begin{pmatrix} 0 & 0 & 2 z_1 -  \sigma(z_1) \end{pmatrix} \begin{bmatrix} 0 \\ 0 \\ \frac{1}{m_3} \end{bmatrix}
& = K(V z_1)^\top \Theta^\top
\end{align*}
 for $z_1 \in Z_1$ and thus~\eqref{eq:K-cons} is satisfied. Furthermore,
\begin{align*}
 &  \| K(Vz_1) - K(Vz_1 + w)\| \\
 & = \left\| \begin{pmatrix} 0 \\ k z_1 \\ 2 z_1 - \sigma(z_1)  \end{pmatrix} - \begin{pmatrix} w_1 \\ k z_1 + w_2 \\ 2 z_1 - \sigma(z_1) + w_3 \end{pmatrix} \right\| = \|w\|,
\end{align*}
which proves~\eqref{eq:K-upper}, and
\begin{align*}
z_1 \Theta K(Vz_1) &= z_1 \frac{1}{m_3} (2 z_1 - \sigma(z_1) )\\
& = \frac{1}{m_3} (2z_1^2 - |z_1|) \ge \frac{1}{m_3} z_1^2
\end{align*}
for $z_1 \in Z_1$ shows~\eqref{eq:K-lower}. Conditions~\eqref{eq:D2-upper}--\eqref{eq:D2-lower} for~$D$ follow analogously for $z_2 \in Z_2 = Z_1$. For~\eqref{eq:D2-Lipschitz} consider
\begin{align*}
\|D(V z_2) \| &= \left\| \begin{pmatrix} 0 \\ d z_2 \\ 2z_2 - \sigma(z_2) \end{pmatrix} \right\| \\
&\le d |z_2| + |2 z_2 - \sigma(z_2) | \\
& \le (d + 2) | z_2| + 1 \le (d + 3) |z_2|, 
\end{align*}
for $z_2 \in Z_2$, 
which shows~\eqref{eq:D2-Lipschitz}. Therefore, via Theorem~\ref{Thm:1+2} we may deduce 
BIBO stability of the internal dynamics of the extended mass on a car system~\eqref{Ex:Wagen}.

\section{Conclusion}
In the present article we elaborated two main results. First, we presented a suitable set of coordinates to represent the internal dynamics of a nonlinear multibody system. 
This representation is completely determined by the system parameters and therefore the internal dynamics can be decoupled algorithmically and are open to e.g. stability analysis. In particular, it is not necessary to compute the Byrnes-Isidori form to obtain an explicit representation of the internal dynamics. 
Second, we gave an abstract stability result for control systems and derived sufficient conditions on the system parameters of a multibody system such that its internal dynamics are bounded-input, bounded-output stable. We highlight that the conditions can be verified in advance and hence it is not necessary to derive the internal dynamics in the first place.
\ \\

Now, further research aims to find similar representations of the internal dynamics of multibody systems with vector relative degree and differential algebraic systems.
\ \\
Moreover, we aim to achieve similar stability results for nonlinear multibody systems with a more complex function~$f$, a state dependent mass matrix, such as e.g. the robotic manipulator arm with a passive joint, and systems with algebraic constraints such as e.g. systems with a kinematic loop.
\paragraph{Acknowledgements} I thank Thomas Berger (University of Paderborn) for many helpful discussions, suggestions and corrections.
Furthermore, I am indebted to Svenja Dr\"ucker (TU Hamburg) for bringing the extended mass on a car example to my attention.
\appendix 
\section{Proof of Theorem~\ref{Thm:V-is-a-Lyapunov-function}} \label{Proof:Thm-V-is-a-Lyapunov-function}
\begin{proof} 
Recall $\Theta = \phi_2 M^{-1} \in \R^{(n-m) \times n}$, and for $\tau := \|\Theta\|$
let $0 < \lambda < 2\kappa/\tau^2 $ and $0 < q < 2 \delta/( d^2 + 2\lambda )$. 
For $i=1,2$ let $\eta_i \in Z_i$ and $y_i \in B_{r_i}(0)$. We calculate the Lie derivative of~$\cV$ from~\eqref{fcn:V} along the vector field~$\cF$ from~\eqref{eq:ID_constant_matrices}
\begin{equation}\label{eq:Lyapunov}
\begin{aligned}
&\cV'(\eta_1,\eta_2) \cdot \cF(\eta_1,\eta_2,y_1, y_2)\\
&=   K(\lambda^{-1} V \eta_1)^\top \Theta^\top  \eta_2 + q \lambda \eta_2^\top \eta_2\\
&+  \eta_2^\top  \Theta \big(-K(\cM y_1 + \lambda^{-1} V \eta_1)  - D(\cM y_2 + V \eta_2) \big) \\
&- \eta_2^\top \Theta C(\cM y_1 + \lambda^{-1} V \eta_1, \cM y_2 + V \eta_2) (\cM y_2 + V \eta_2) \\
&+ q \eta_1^\top \Theta \big(-K(\cM y_1 + \lambda^{-1} V \eta_1) - D(\cM y_2 + V \eta_2) \big) \\
& - q \eta_1^\top \Theta C(\cM y_1 + \lambda^{-1} V \eta_1, \cM y_2 + V \eta_2) (\cM y_2 + V \eta_2) \\
&= 
\eta_2^\top \Theta \big( K(\lambda^{-1} V \eta_1) -K(\cM y_1 + \lambda^{-1} V \eta_1)  \big) \\
& + q \lambda \| \eta_2 \|^2 
- \eta_2^\top \Theta D(\cM y_2 + V \eta_2) \\
& - \eta_2^\top \Theta C(\cM y_1 + \lambda^{-1} V \eta_1, \cM y_2 + V \eta_2) (\cM y_2 + V \eta_2) \\
&- q \eta_1^\top  \Theta  K(\cM y_1 + \lambda^{-1}V \eta_1)  
- q \eta_1^\top  \Theta D(\cM y_2 + V \eta_2) \\
&- q \eta_1^\top \Theta C(\cM y_1 + \lambda^{-1} V \eta_1, \cM y_2 + V \eta_2) (\cM y_2 + V \eta_2). \\
\end{aligned}
\end{equation}
For purpose of better legibility we set $\mu := \|\cM\|$ and estimate the addends in~\eqref{eq:Lyapunov} separately for $\eta_i \in Z_i$, resp. 
Note, that since~$V$ from Lemma~\ref{Lem:phi2} has orthonormal columns we have~$\|V z\| = \|z\|$ for~$z \in \R^{n-m}$. 
\\
\underline{Step i)}
\begin{equation*}
\begin{aligned}
&\eta_2^\top \Theta \big( K(\lambda^{-1} V \eta_1) -K(\cM y_1 + \lambda^{-1} V \eta_1)  \big)  \\
&\le \tau \| \eta_2 \| \|  K(\lambda^{-1} V \eta_1) -K(\cM y_1 + \lambda^{-1} V \eta_1)   \| \\
&\overset{\eqref{eq:K-upper}}{\le}  \tau \| \eta_2 \| g_1(\cM y_1)  \le  \tau \tilde K \| \eta_2 \|
\end{aligned}
\end{equation*}
\underline{Step ii)}
\begin{equation*}
\begin{aligned}
&-\eta_2^\top \Theta D(\cM y_2 + V \eta_2) \\
&= \eta_2^\top \Theta \big( D(V \eta_2) - D(\cM y_2 + V \eta_2) - D(V \eta_2) \big) \\
& \le \tau  \| \eta_2\| \|  D(V \eta_2) - D(\cM y_2 + V \eta_2)\| - \eta_2^\top  \Theta D(V \eta_2) \\
& \overset{\eqref{eq:D2-upper}}{\le} \tau  \| \eta_2\| g_{2}(\cM y_2) - \eta_2^\top  \Theta D(V \eta_2) 
 \overset{\eqref{eq:D2-lower}}{\le} \tau \tilde D  \| \eta_2 \| - \delta \| \eta_2\|^2
\end{aligned}
\end{equation*}
\underline{Step iii)}
\begin{equation*}
\begin{aligned}
&- q \eta_1^\top \Theta K(\cM y_1 + \lambda^{-1}V \eta_1) \\
& \le q \tau  \| \eta_1\| \| K(\lambda^{-1}V \eta_1) - K(\cM y_1 + \lambda^{-1}V \eta_1)\| \\ 
& \quad -  q \eta_1^\top \Theta K(\lambda^{-1}V \eta_1)  \\
&\overset{\eqref{eq:K-upper}}{\le} q \tau  \| \eta_1\| g_1(\cM y_1) -  q \eta_1^\top \Theta K(\lambda^{-1}V \eta_1) \\
&\overset{\eqref{eq:K-lower}}{\le} q \tau \tilde K  \| \eta_1\| - q  \kappa  \lambda^{-1} \|\eta_1\|^2 
\end{aligned}
\end{equation*}
\underline{Step iv)}
\begin{equation*}
\begin{aligned}
& -q \eta_1^\top \Theta D(\cM y_2 + V \eta_2)   \le q \tau \| \eta_1 \| \| D(\cM y_2 + V \eta_2) \|  \\
& \overset{\eqref{eq:D2-Lipschitz}}{\le} q  \|\eta_1\| d \|\cM y_2 + V \eta_2 \|  \le q \tau d \mu r_2 \|\eta_1\| + q \tau d \|\eta_1\| \|\eta_2\| \\
& \le q \tau d \mu r_2 \|\eta_1\| +  q\frac{\tau^2}{2}\|\eta_1\|^2 + q \frac{d^2}{2} \| \eta_2 \|^2,
\end{aligned}
\end{equation*}
where we used $2 a b \le a^2 + b^2$ for all $a,b \in \R$ in the last line. 
We continue\\
\underline{Step v)}
\begin{equation*}
\begin{aligned}
& -q \eta_1^\top \Theta C(\cM y_1 + \lambda^{-1} V \eta_1, \cM y_2 + V \eta_2) (\cM y_2 + V \eta_2) \\
 = & \, q \eta_1^\top \Theta \big[ C( \lambda^{-1} V \eta_1, \cM y_2 + V \eta_2) \\ 
& \quad - C(\cM y_1 + \lambda^{-1} V \eta_1, \cM y_2 + V \eta_2) \big] (\cM y_2 + V \eta_2) \\
& - q \eta_1^\top \Theta C( \lambda^{-1} V \eta_1, \cM y_2 + V \eta_2) (\cM y_2 + V \eta_2)\\
\le & \, q \tau \| \eta_1\| \| \big[ C( \lambda^{-1} V \eta_1, \cM y_2 + V \eta_2) \\ 
& \quad - C(\cM y_1 + \lambda^{-1} V \eta_1, \cM y_2 + V \eta_2)  \big] (\cM y_2 + V \eta_2) \| \\ 
& - q \eta_1^\top \Theta C( \lambda^{-1} V \eta_1, \cM y_2 + V \eta_2)(\cM y_2 + V \eta_2) \\
\overset{\eqref{eq:C-upper1}}{\le}  & \, q \tau \|\eta_1\| g_3(\cM y_1) a_3(\cM y_2 + V \eta_2) \\
& - q \eta_1^\top \Theta C( \lambda^{-1} V \eta_1, \cM y_2 + V \eta_2)(\cM y_2 + V \eta_2) \\
\overset{\eqref{eq:C-lower1}}{\le}  & \, q \tau \tilde C_3 \| \eta_1\| a_3( \cM y_2 + V \eta_2 )  
- \frac{q }{\lambda} \| \eta_1 \|^2 \, b_3( \cM y_2 + V \eta_2 )   \\
\le & q \tau \tilde C_3 \| \eta_1\| a_3( \cM y_2 + V \eta_2 )  - \frac{q }{\lambda} \| \eta_1 \|^2 \, a_3( \cM y_2 + V \eta_2 )   \\
= & \, \frac{q }{\lambda} \left( \lambda \tau \tilde C_3 - \|\eta_1\| \right) \|\eta_1 \| a_3( \cM y_2 + V \eta_2 ) 
\end{aligned}
\end{equation*}
%
\underline{Step vi)}
\begin{equation*}
\begin{aligned}
& - \eta_2^\top \Theta C(\cM y_1 + \lambda^{-1} V \eta_1, \cM y_2 + V \eta_2) (\cM y_2 + V \eta_2) \\
 = & \, \eta_2^\top \Theta \big[ C(\cM y_1 + \lambda^{-1} V \eta_1,  V \eta_2) \\ 
& \quad - C(\cM y_1 + \lambda^{-1} V \eta_1, \cM y_2 + V \eta_2) \big] (\cM y_2 + V \eta_2) \\
& - \eta_2^\top \Theta C(\cM y_1 + \lambda^{-1} V \eta_1,  V \eta_2) (\cM y_2 + V \eta_2) \\
\le & \, \tau \| \eta_2\| \| \big[ C( \cM y_1 + \lambda^{-1} V \eta_1, V \eta_2) \\ 
& \quad - C(\cM y_1 + \lambda^{-1} V \eta_1, \cM y_2 +  V \eta_2)  \big] (\cM y_2 + V \eta_2) \| \\ 
& -  \eta_2^\top \Theta C(\cM y_1 + \lambda^{-1} V \eta_1, V \eta_2)(\cM y_2 + V \eta_2) \\
\overset{\eqref{eq:C-upper2}}{\le}  & \, \tau \|\eta_2\| \, g_4(\cM y_2) \, a_4(\cM y_1 + \lambda^{-1} V \eta_1 ) \, \| \eta_2\| \\
& -  \eta_2^\top \Theta C( \cM y_1 + \lambda^{-1} V \eta_1,  V \eta_2)(\cM y_2 + V \eta_2) \\
\overset{\eqref{eq:C-lower2}}{\le}  & \, \tau \|\eta_2\|^2 \, g_4(\cM y_2) \, a_4(\cM y_1 + \lambda^{-1} V \eta_1 ) \\
& -  b_4(\cM y_1 + \lambda^{-1} V \eta_1 ) \| \eta_2 \| \| V \eta_2\|^2 \\
& -  \eta_2^\top \Theta C( \cM y_1 + \lambda^{-1} V \eta_1,  V \eta_2) \cM y_2  \\
\overset{\eqref{eq:C-Lipschitz}}{\le} & \, \tau \tilde C_4 \,  a_4(\cM y_1 + \lambda^{-1} V \eta_1 ) \|\eta_2\|^2 \\
& -  b_4( \cM y_1 + \lambda^{-1} V \eta_1 ) \, \| \eta_2 \|^3 \\
& + \tau  \mu r_2 \,  a_4( \cM y_1 + \lambda^{-1} V \eta_1 ) \, \|\eta_2 \|^2 \\
\le & \,   \left( \tau (\tilde C_4 +  \mu r_2)  - \|\eta_2\| \right)  \|\eta_2\|^2 \, a_4( \cM y_1 + \lambda^{-1} V \eta_1 ),
\end{aligned}
\end{equation*}
where we used $\| V z\| = \|z\|$ for all~$z \in \R^{n-m}$ in the second last estimation.
We summarize the calculations above to estimate $ \cV'(\eta_1,\eta_2) \cdot \cF(\eta_1, \eta_2, y_1, y_2)$ for $\eta_i \in Z_i$ and $y_i \in B_{r_i}(0)$, $i=1,2$
\begin{equation*}
\begin{aligned}
& \cV'(\eta_1,\eta_2) \cdot \cF(\eta_1, \eta_2, y_1, y_2)  \\
& \le
q \lambda \| \eta_2 \|^2 
+  \tau \tilde K  \| \eta_2 \| 
+ \tau \tilde D  \| \eta_2 \| - \delta \| \eta_2\|^2 +
 q \tau \tilde K  \| \eta_1\| \\
 & - q  \frac{\kappa}{  \lambda } \| \eta_1\|^2 
 + q \tau d \mu r_2 \|\eta_1\| +  q \frac{\tau^2}{2}\|\eta_1\|^2 
 + q\frac{d^2}{2} \| \eta_2 \|^2 \\
& + \frac{q }{\lambda} \left( \lambda \tau \tilde C_3 - \|\eta_1\| \right) \|\eta_1 \| a_3( \cM y_2 + V \eta_2 ) \\
& +  \left( \tau (\tilde C_4 +  \mu r_2)  - \|\eta_2\| \right)  \|\eta_2\|^2 \, a_4( \cM y_1 + \lambda^{-1} V \eta_1 ) .
\end{aligned}
\end{equation*}
Sorting these expressions and inserting the constants from~\eqref{def:constants} yields 
\begin{equation} \label{eq:Lyapunov-short}
\begin{aligned}
 \cV' & (\eta_1, \eta_2) \cdot \cF(\eta_1, \eta_2, y_1, y_2) \\ 
 \le & - \varepsilon_1 \| \eta_1\|^2 +  E_1 \| \eta_1 \| - \varepsilon_2 \|\eta_2\|^2 + E_2 \| \eta_2\| \\
& -  \frac{q }{\lambda} ( \|\eta_1\| - \gamma_1 )  \|\eta_1 \| a_3( \cM y_2 + V \eta_2 ) \\
& -    \left(  \|\eta_2\| - \gamma_2 \right)  \|\eta_2\|^2 \, a_4( \cM y_1 + \lambda^{-1} V \eta_1 ),
\end{aligned}
\end{equation}
where~$\varepsilon_1, \varepsilon_2 > 0$ and $ E_1, E_2 \ge 0$ via the choice of~$q$ and~$\lambda$, 
and $\gamma_1, \gamma_2 \ge 0 $. 
Now, we consider the function
\begin{align} \label{eq:compare-function}
 W: \R^{2(n-m)} & \to \R \\
( w_1 , w_2 ) &\mapsto -\ve_1 \|w_1\|^2 + E_1 \|w_1\| -\ve_2 \|w_2\|^2 + E_2 \|w_2\| \nonumber, 
\end{align}
with $\ve_i >0$, $E_i \ge 0$ for~$i=1,2$ as above.
A short calculation yields that for~$w_i \in \tilde Z_i$ we have
$
 W(w_1,w_2) \le 0,
$ 
and $-(\|w_1\| - \gamma_1) < 0$ and $-(\|w_2\| - \gamma_2) < 0$.
Comparing~\eqref{eq:Lyapunov-short} and~$\eqref{eq:compare-function}$ yields assertion~\eqref{eq:Lyapunov_leq_0}
\begin{equation*}
\cV'(\eta_1,\eta_2) \cdot \cF(\eta_1, \eta_2, y_1, y_2) \le W(\eta_1, \eta_2) \leq 0,
\end{equation*}
for all $\eta_i \in Z_i \cap \tilde Z_i$, $i=1,2$, and $y_1 \in B_{r_1}(0),\, y_2 \in B_{r_2}(0)$.
\end{proof}

\end{document}